\documentclass[review,hidelinks,onefignum,onetabnum,12pt]{article}
\usepackage{amsmath}    
\usepackage{amssymb}
\usepackage{graphicx}   
\graphicspath{{figures/}}
\usepackage{verbatim}   
\usepackage{color}      
\usepackage{hyperref}   
\usepackage{enumerate}
\usepackage{listings}
\usepackage{xcolor}
\usepackage{slashed}
\usepackage{amsthm}
\usepackage{amsmath}    
\usepackage{amssymb}
\usepackage{subfigure}
\usepackage{mathrsfs}
\usepackage{bbold}
\usepackage{adjustbox}
\usepackage[british]{babel}
\usepackage{tikz}
\usepackage{bm}
\usetikzlibrary{arrows.meta}
\usepackage[margin=1in,footskip=0.3in]{geometry}

\usepackage[final]{showlabels}

\renewcommand{\phi}{\varphi}

\renewcommand{\and}{\wedge}

\newcommand{\vb}[1]{\mathbf{#1}}

\definecolor{mygreen}{rgb}{0,0.6,0}
\definecolor{mygray}{rgb}{0.5,0.5,0.5}
\definecolor{mymauve}{rgb}{0.58,0,0.82}

\newtheorem{proposition}{Proposition}[section]

\newtheorem{remark}{Remark}[section]

\begin{document}


\title{Minimization-based embedded boundary methods as polynomial corrections: a stability study of discontinuous Galerkin for hyperbolic equations}
%
\author{Mirco Ciallella\thanks{Laboratoire Jacques-Louis Lions, Universit\'e Paris Cit\'e, CNRS, UMR 7598, Paris, France 
  (\href{mailto:mirco.ciallella@u-paris.fr}{mirco.ciallella@u-paris.fr}).}
}

\date{}
%
%
%
%
%
\maketitle

\begin{abstract} 
This work establishes a novel, unified theoretical framework for a class of high order embedded boundary methods, revealing that the Reconstruction for Off-site Data (ROD) treatment shares a fundamental structure with the recently developed shifted boundary polynomial correction [Ciallella, M., et al. (2023)]. By proving that the ROD minimization problem admits an equivalent direct polynomial correction formulation, we unlock two major advances. First, we derive a significant algorithmic simplification, replacing the solution of the minimization problem with a straightforward polynomial evaluation, thereby enhancing computational efficiency. Second, and most critically, this reformulation enables the first stability result for the ROD method when applied to the linear advection equation with discontinuous Galerkin discretization. Our analysis, supported by a comprehensive eigenspectrum study for polynomial degrees up to six, characterizes the stability region of the new ROD formulation. The theoretical findings, which demonstrate the stability constraints, are validated through targeted numerical experiments.
\end{abstract}



\section{Introduction}

high order discontinuous Galerkin (DG) methods \cite{reed1973triangular,cockburn2001runge} offer a powerful framework for achieving accurate solutions with improved computational efficiency \cite{wang2013high}. However, their application to problems with complex geometries presents a significant challenge, as the accuracy of the final numerical solution is intrinsically linked to the precision of the geometry representation \cite{bassi1997high}. The isoparametric approach, which uses high order polynomial mappings to represent curved boundaries, can provide optimal convergence \cite{bassi1997high}. However, the practical difficulties associated with generating and validating high order curvilinear meshes remain relevant hurdles for realistic applications \cite{persson2009curved,coppeans2025aerodynamic}: nonlinear mappings, complex quadrature formulas, and mesh quality control are the major issues to be addressed.

Embedded (or immersed) boundary methods present an alternative paradigm, aiming to circumvent these mesh generation challenges by using simple, fixed background grids over which complex boundaries can move freely. This shift transfers the primary difficulty from mesh generation to the formulation of accurate and stable boundary conditions, as the method introduces a geometric error proportional to the mesh size $h$. Following the pioneering immersed boundary work of Peskin \cite{peskin2002immersed}, extensive development has produced many methods for fluid and solid mechanics \cite{fadlun2000combined,hansbo2002unfitted,fedkiw1999non,mittal2005immersed,berger2012progress,monasse2012conservative,coco2013finite,rapaka2020efficient}. While historically many were low-order, recent years have seen growing interest in high order embedded treatments compatible with high order solvers \cite{tan2010inverse,atallah2022high,giuliani2022two,ciallella2023shifted,ciallella2022extrapolated,BOSCHERI2025114215,visbech2025spectral}, which better adapt to modern computer architectures.

This work advances this field by addressing a specific gap: the stability of \emph{minimization-based} high order embedded methods, whose analysis has been hindered by their implicit formulation. We focus on the Reconstruction for Off-site Data (ROD) method \cite{costa2018very}, initially developed to impose high order conditions for curved boundaries on simplicial meshes within finite volume \cite{costa2018very,costa2019very} and DG frameworks \cite{santos2024very,ciallella2024very}. In a DG context, ROD operates by constructing a modified boundary polynomial via a constrained minimization problem, which is as close as possible to the internal polynomial and exactly satisfies the true boundary condition. This modified polynomial is then used to impose consistent boundary conditions on a computational boundary, which does not match the real boundary.
The central analytical obstacle is that this minimization-based definition cannot be directly employed for stability investigation. Therefore, our first and fundamental contribution is to prove that the ROD formulation is mathematically equivalent to an explicit polynomial correction, akin in structure to the shifted boundary polynomial correction \cite{ciallella2023shifted,ciallella2025stability}. This equivalence reveals ROD to be part of a broader unified framework for embedded treatments. Practically, it yields a significant algorithmic simplification, replacing the solution of a minimization problem with the application of a precomputable correction.
This reformulation enables our second contribution: a rigorous stability analysis of high order DG methods coupled with the ROD treatment for the one-dimensional linear advection equation
\begin{equation}
    \label{eq:linear_advection_intro}
    \begin{cases}
    \partial_t u(x,t) + \partial_x u(x,t) = 0, \quad & x\in[x_L,x_R], \; t\geq 0 \\
    u(x_L,t) = u_D(t),\quad & t\geq 0, \\
    u(x,0) = u_0(x),\quad & x\in[x_L,x_R],
    \end{cases}
\end{equation}
where $u$ is the conservative variable and the advection speed is equal to one.
By recasting the method in polynomial correction form, we can straightforwardly construct the discretized operator and study its eigenspectrum, an established technique for analyzing complex high order boundary treatments \cite{M2AN_2015__49_1_39_0,yang2025inverse,ciallella2025stability}, which is made even easier by the local structure of DG. We extend the analysis beyond the standard ROD formulation, based on the Euclidean norm, to consider an $L^2$-norm variant, demonstrating how the choice of norm influences stability.
Our findings reveal that, for polynomial degrees $p \leq 4$, a stable explicit scheme can be constructed under a standard CFL condition, an improvement over related embedded methods \cite{ciallella2025stability}. For very high orders, we characterize the stability limitations and show that the implicit time integration can improve them, although with specific CFL constraints: in particular, we show that the CFL has to be high enough to ensure stability, which does not make it optimal if one studies unsteady problems. 
All theoretical results are validated through comprehensive numerical experiments spanning multiple boundary configurations and orders of accuracy from second to seventh.

The rest of the paper is organized as follows.
In section \ref{sec:DGadvection} the high order DG discretization for the linear advection equation is presented. 
In section \ref{sec:embedded} the notation of the embedded boundary methods is introduced.
Section \ref{sec:SBcorrection} describes the high order SB polynomial correction, which is used to make a link with the novel ROD polynomial corrections.
In section \ref{sec:RODcorrection}, we show that several variations of the ROD minimization-based method can be recast as a polynomial correction, similar to the one described in section \ref{sec:SBcorrection}. 
The stability analysis for all embedded boundary methods are presented in section \ref{sec:stability}.
The analysis is applied first to DG with standard periodic boundary conditions in section \ref{sec:periodic} to show that standard CFL conditions are retrieved. 
Then, the DG methods with the considered embedded boundary method and explicit time integration are studied in 
section \ref{sec:RODexplicit};
while their coupling with implicit time integration are considered in 
section \ref{sec:RODimplicit}.
Numerical tests that validate all the theoretical results are presented in section \ref{sec:tests},
and some conclusions are drawn in section \ref{sec:conclusions}.

\section{Discontinuous Galerkin for linear advection}\label{sec:DGadvection}

This section introduces the one-dimensional DG formulation that serves as the foundation for the analysis of the embedded boundary treatments considered in this work.
Consider a one-dimensional computational domain $\Omega$ partitioned into $N_e$ non-overlapping cells $\Omega_i$, such that $\Omega=\bigcup_{i=1}^{N_e} \Omega_i$.
For simplicity, each cell $\Omega_i=[x_{i-1/2},x_{i+1/2}]$ has a fixed size of $\Delta x$. 
The discrete solution $u_h$ resides in the broken polynomial space
\begin{equation}
V_h:= \bigoplus_{i=1}^{N_e} V_i, \qquad \quad V_i := \mathbb{P}^p(\Omega_i),
\end{equation}
where $V_i$ denotes the local space of polynomials of degree $p\geq 0$ on cell $\Omega_i$. Thus, $u_h$ is a piece-wise polynomial that may be discontinuous at cell interfaces.

Within each element $\Omega_i$, the numerical solution $u_h$ is expressed in terms of local basis functions $\{\phi_n\}_{n=0}^{p}$ of the space $V_i$:
\begin{equation}
	\label{eq:discontinuous_approximation}
    u_h(x, t)|_i = \sum_{n=0}^{p} \phi^i_n(x) u_{i,n}(t) ,
\end{equation}
where $u_{i,n}(t)$ are the time-dependent modal coefficients of cell $\Omega_i$.

To be aligned with the numerical experiments, the model problem under consideration is the linear advection equation with a source term:
\begin{equation}
    \label{eq:linear_advection_source}
    \partial_t u(x,t) + \partial_x u(x,t) = s(x).
\end{equation}
which will be used to derive the semi-discrete form. For brevity, the explicit dependence on $x$ and $t$ is often omitted below.

The DG weak formulation in the semi-discrete form for cell $\Omega_i$ is obtained by projecting \eqref{eq:linear_advection_source} and applying integration by parts, 
\begin{equation}
	\label{eq:weak_formulation_integrated}
	\int_{\Omega_i} \phi^i_m \frac{\text{d} u_h}{\text{d} t}  \,\text{d}x - \int_{\Omega_i} \partial_x  \phi^i_m u_h \,\text{d}x + 
    \int_{\partial \Omega_i} \phi^i_m \hat u(u_h^-,u_h^+) n \,\text{d}\Sigma = \int_{\Omega_i} \phi^i_m s(x) \,\text{d}x , \quad\forall \phi^i_m \in V_i,
\end{equation}
where $n$ is the outward normal. Here, $u_h^-$ and $u_h^+$ denote the left and right traces at an interface, and $\hat{u}(\cdot,\cdot)$ is a numerical flux. For the scalar advection problem with unit speed, the upwind flux $\hat{u}(u_h^-, u_h^+) = u_h^-$ is employed.
In one-dimension, equation \eqref{eq:weak_formulation_integrated} can be simply rewritten as
\begin{equation}
	\label{eq:weak_formulation_flux}
	\int_{\Omega_i} \phi^i_m \frac{\text{d} u_h}{\text{d} t}  \,\text{d}x - \int_{\Omega_i} \partial_x \phi^i_m u_h \,\text{d}x + 
     \bigl[\phi^i_m \hat u(u_h^-,u_h^+)\bigr]^{i+1/2}_{i-1/2} = \int_{\Omega_i} \phi^i_m s(x) \,\text{d}x , \quad\forall \phi^i_m \in V_i.
\end{equation}

Substituting the expansion \eqref{eq:discontinuous_approximation} into \eqref{eq:weak_formulation_flux} yields the semi-discrete system for element $\Omega_i$ in matrix-vector form: 
\begin{equation}
	\label{eq:discontinuous_galerkin_matrix_form}
	M_i  \frac{\text{d} \mathbf{u}_i}{\text{d} t} - K^s_i \mathbf{u}_i + K^R_i \mathbf{u}_i - K^L_i \mathbf{u}_{i-1} = \mathbf{s}_i, 
\end{equation}
where  $\mathbf{u}_i=[u_{i,0}, \ldots, u_{i,p}]^T$ is the vector of local degrees of freedom.
The elemental mass matrix $M_i$, the stiffness matrix $K^s_i$, and the interface flux matrices are defined as,
\begin{align}
    M_i = \int_{\Omega_i} \phi^i_m \phi^i_n \,\text{d}x ,&\qquad K^s_i = \int_{\Omega_i} \partial_x \phi^i_m \phi^i_n \,\text{d}x , \label{eq:mass_stiff_mat} \\
    K^R_i =  \phi_m^i(x_{i+1/2}) \phi_n^i(x_{i+1/2})  ,&\qquad K^L_i = \phi_m^i(x_{i-1/2}) \phi_n^{i-1}(x_{i+1/2}),      \label{eq:interface_mat}
\end{align}
where $\phi_m^i(x_{i+1/2})$ represents the $m$-th basis of cell $\Omega_i$ in $x_{i+1/2}$.

The stability analysis is performed by examining the eigenspectrum of the semi-discrete system. To proceed, we first assemble the local DG formulation into a global system of ordinary differential equations:
\begin{equation}\label{eq:semidiscrete} 
    \mathbf{M}\frac{\text{d} \mathbf{U}}{\text{d} t}  = \mathbf{K} \mathbf{U}, 
\end{equation} 
where $\mathbf{M}$ is the block-diagonal global mass matrix, $\mathbf{K}$ is the global stiffness matrix incorporating both interior stiffness and boundary conditions, and $\mathbf{U}=[\mathbf{u}_1,\ldots,\mathbf{u}_{N_e}]^T$ is the global vector of all modal coefficients.

The stability of the fully discrete scheme depends on the interaction between spatial discretization and time integration.
We employ arbitrary high order Deferred Correction (DeC) schemes \cite{dutt2000spectral,micalizzi2024new} of order $p+1$ for time marching. For a linear system like \eqref{eq:semidiscrete}, the amplification factor of a $(p+1)$-th order DeC scheme can be expressed as
\begin{equation}\label{eq:amplificationRK} 
    \mathbf{U}^{n+1} = \left( \sum_{k=0}^{p+1} \frac{\mu^k}{k!}  \right)\mathbf{U}^n, 
\end{equation}
where $\mathbf{U}^n = \mathbf{U}(t^n)$, $\lambda$ denotes the eigenvalue of semi-discrete operator $\mathbf{M}^{-1}\mathbf{K}$, and $\mu=\lambda \Delta t$. 
In this study, spatial polynomial degrees up to $p=6$ are investigated, coupled with seventh-order DeC time integration. Consequently, the stability region in the complex $\mu$-plane is determined by the condition $| \sum_{k=0}^{p+1} \mu^k / k! | \leq 1$.

To perform the analysis is therefore needed to recast the ROD embedded boundary treatment in the general matrix form $K\mathbf{u}$. Although this is trivial for the SB polynomial correction as it was previously shown in \cite{ciallella2025stability} and recalled in section \ref{sec:SBcorrection}, this is not straightforward for the minimization-based ROD techniques which need careful manipulation to be written in such a form. 
In section \ref{sec:RODcorrection}, the ROD minimization problem is explicitly solved to show that this class of methods as well can be written as a special polynomial correction. This allows one to make a link between the two families of methods and, more interestingly allows one to perform a stability analysis.

\section{Notation of embedded boundary  methods} \label{sec:embedded}

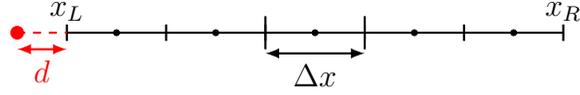
\begin{figure}
\centering
    \begin{tikzpicture}[>=latex,scale=1.1]

    \def\n{5} 
    \def\dx{1.2} 

    \foreach \i in {0,...,\n}{
    \draw[thick] (\i*\dx,-0.1) -- (\i*\dx,0.1);
    }

    \draw[thick] (0,0) -- (\n*\dx,0);
    \foreach \i in {1,...,\n}{
    \filldraw[black] ({(\i-0.5)*\dx},0) circle (1pt);
    }
    \draw[thick] (3*\dx,-0.2) -- (3*\dx,0.2);
    \draw[thick] (2*\dx,-0.2) -- (2*\dx,0.2);

    \node[below] at (0,0.5) {$x_L$};
    \node[below] at (6,0.5) {$x_R$};


    \draw[thick,red,dashed] (-0.6,0) -- (0,0);
    \draw[thick,fill=red,red] (-0.6,0) circle (2pt);
    \draw[thick,<->,red] (-0.6,-0.2) -- (0,-0.2);
    \node[below,red] at (-0.3,-0.2) {$d$};
    \draw[thick,<->] (2*\dx,-0.25) -- (3*\dx,-0.25);
    \node[below] at (2.5*\dx,-0.25) {$\Delta x$};

    \end{tikzpicture}
    \caption{One-dimensional domain with an embedded boundary condition (red circle) imposed at a distance $d$ from $x_L$, which is defined as the computational boundary in this simplified configuration.}\label{fig:1dmesh}
\end{figure}

The work analyzes the stability of the ROD minimization-based treatment in the standard unfitted configuration where the physical boundary is separated from the computational mesh by a distance $d$ of the order of the mesh size $\Delta x$.  
This scenario is typical for {\it immersed} boundary simulations that employ a Cartesian background grid, as illustrated in figure \ref{fig:1dmesh}.
While in multi-dimensional settings the correspondence between a point on the true boundary and its projection onto the computational boundary may be defined in various ways (e.g., closest-point or level-set), the one-dimensional case admits a unique and straightforward mapping. This simplification permits a focused investigation of the stability properties introduced by the minimization-based boundary treatment.

Thus, $\bar{\mathbf{x}}$ denotes a point on the physical boundary, and $\tilde{\mathbf{x}}$ its associated point on the computational boundary. 
The straightforward mapping between the two is simply given by $\bar{\mathbf{x}} = \tilde{\mathbf{x}} + \mathbf{d}$, where $\mathbf{d}$ is the distance vector.
Considering $\mathbf{n}_e$ the unit outward normal of the boundary face, the signed distance can be defined as the scalar product: $d = - \mathbf{n}_e^T (\bar{\mathbf{x}} - \tilde{\mathbf{x}} )$.
Following this notation, the distance can vary within  $[-\Delta x, \Delta x]$.
A positive $d$ indicates that the true boundary lies inside the first mesh element, whereas a negative $d$ places it outside the element. In one dimension, the mapping reduces to the scalar relation $\bar{x} = \tilde{x} + d$.

At the computational boundary, the embedded condition is enforced via a numerical flux. For the upwind flux considered here, the left-boundary flux is $\hat u (v_h, u_h^+) = v_h$, where $v_h$ represents the high order boundary reconstruction provided by the ROD treatment.

\section{Shifted boundary polynomial correction}\label{sec:SBcorrection}

This section briefly reviews the SB polynomial correction \cite{ciallella2023shifted} to establish notation and provide context for the ROD formulations derived in section \ref{sec:RODcorrection}. A comprehensive stability analysis of the SB method is available in \cite{ciallella2025stability} and is not repeated here for brevity.
The SB approach enforces high order boundary conditions on the computational boundary $\tilde{x}$ by applying a corrected Dirichlet value thanks to a truncated Taylor expansion.
Let $u_D$ be the Dirichlet boundary condition on the real boundary $\bar x$, which does not coincide with $x_L$.
A corrected boundary condition $v_h$ for the computational boundary $\tilde x$ can be constructed starting from the Taylor expansion:
\begin{equation}\label{eq:taylor}
    u(\bar x)=u(\tilde x + d) = u(\tilde x) + \sum_{m=1}^p \partial_x^{(m)} u(\tilde x) \frac{d^m}{m!}.
\end{equation}
Hence, a modified boundary condition can be developed including all high order derivative terms and 
by enforcing the prescribed boundary condition $u(\bar x) = u_D$:
\begin{equation}
    v_h(\tilde x) = u_D - \sum_{m=1}^p \partial_x^{(m)} u_h(\tilde x) \frac{d^m}{m!}.  
\end{equation}
As shown in \cite{ciallella2023shifted}, from equation \eqref{eq:taylor} it is direct to infer that the sum of derivative terms in the above expression is equivalent to the difference of the polynomial approximation evaluated at the two boundary locations.
Therefore, the SB Dirichlet boundary condition $v_h$ on $\tilde x$ can be rewritten using the following correction:
\begin{equation}\label{eq:SBcorrection}
\begin{split}
    v_h(\tilde x) &= u_D - (u_h(\tilde x + d) - u_h(\tilde x)) \\
                  &= u_h(\tilde x) - u_h(\bar x) + u_D \\
                  &= [\bm{\phi}^{\text{SB}}(\tilde x)]^T\, \mathbf{u} + u_D ,
\end{split}
\end{equation}
where $\mathbf{u}$ are the modal coefficients of the boundary element. 
Here, we also introduce the modified SB basis $\bm{\phi}^{\text{SB}}$ as:
\begin{equation}
    \bm{\phi}^{\text{SB}}(\tilde x) = \bm{\phi}(\tilde x) - \bm{\phi}(\bar x), 
\end{equation}
where $\bm{\phi} = [\phi_0,\ldots,\phi_p]^T$.

It should be noticed that equation \eqref{eq:SBcorrection} also holds in multiple dimensions since the Dirichlet condition imposed on a quadrature point is not influenced by the BC imposed at the next point. Therefore, in multiple space dimensions, we would simply have:
$$ v_h(\mathbf{\tilde x}) = [\bm{\phi}^{\text{SB}}(\mathbf{\tilde x})]^T\, \mathbf{u} + u_D, $$
where $\mathbf{x} = (x, y)$ and $\bm{\phi}(\mathbf{x})$ are now the two-dimensional basis functions.

\section{ROD as a polynomial correction} \label{sec:RODcorrection}

The Reconstruction for Off-site Data (ROD) method is a minimization-based high order embedded boundary treatment that can be used to impose Dirichlet boundary conditions on the computational boundary.
The main idea of this method consists in reconstructing the solution at the computational boundary by using a constrained minimization problem that creates a modified polynomial $v_h$ that is as close as possible to the internal polynomial $u_h$ and strictly satisfies the Dirichlet condition on the real boundary.

Considering the one-dimensional configuration, the ROD correction can be summarized as follows.
The same notation for the real boundary $\bar x$ and the computational boundary $\tilde x$ is used as in section \ref{sec:SBcorrection}.
Let $u_h$ be the polynomial approximation of degree $p$ in the boundary element.
A modified polynomial $v_h$ is defined as:
\begin{equation}
    v_h(x) = \sum_{m=0}^p \phi^i_m(x) v_{m} = \bm{\phi}^T(x) \vb{v},
\end{equation}
where the coefficients $\mathbf{v} = [v_{0}, \ldots, v_{p}]^T$ are determined by minimizing the distance to $\mathbf{u}$  subject to the boundary constraint $v_h(\bar x)=u_D$. 

More generally the ROD reconstruction consists in finding the polynomial coefficients $\mathbf{v}$ by minimizing the cost function $\mathcal{J}(\mathbf{v}) = \mathcal{D}(\mathbf{v}, \mathbf{u})$ subject to linear constraints $\mathcal{G}(\mathbf{v})=0$ (the boundary conditions).

It can be noticed that the distance $\mathcal{D}$ can be chosen in different ways as long as it stays a convex function.
The simplest choice, as also followed in \cite{santos2024very,ciallella2024very}, constists in taking the Euclidean distance, meaning 
$$ \mathcal{D}^E(\mathbf{v}, \mathbf{u}) =  \frac12 \|\mathbf{v} - \mathbf{u} \|_2^2. $$
The ROD method using the $\mathcal{D}^E$ will be referred to as ROD-E. 

Herein, we also show a possible alternative to the ROD-E, which is retrieved by taking the distance based on the $L^2$ norm that reads
$$ \mathcal{D}^{L^2}(\mathbf{v}, \mathbf{u}) =  \frac12 \|\mathbf{v} - \mathbf{u} \|_{L^2}^2. $$
For this formulation, the method label will be ROD-$L^2$.

In the following, we show first how it is possible to write the ROD method as a polynomial correction when only one boundary constraint is imposed. 
Finally, we show that it is possible to obtain a similar formulation also when multiple boundary constraints are imposed on the polynomial $v_h$, as it was done in \cite{ciallella2024very,santos2024very}.

\subsection{Euclidean norm as cost function}\label{sec:RODE}

In this section, we start by recasting the ROD-E method, based on the Euclidean distance, as a polynomial correction. 
In this case, the cost function to be minimized reads:
\begin{equation}
    \mathcal{J}(\mathbf{v}) = \mathcal{D}^E(\mathbf{v}, \mathbf{u}) = \frac12 \sum_{j=0}^p (v_j - u_j)^2 .
\end{equation}
The new ROD polynomial $v_h$ must also ensure that the Dirichlet condition on the real boundary is satisfied, meaning $v_h(\bar x) = u_D$.

The constrained minimization problem can be solved by introducing a Lagrange multiplier $\lambda$ and by considering the following functional:
\begin{equation}
    \mathcal{L}(\mathbf{v} , \lambda) = \mathcal{J}(\mathbf{v} ) + \lambda ( v_h(\bar x) - u_D ) .
\end{equation}
The optimality conditions for this problem, in vector form, can be written as
\begin{equation}
    \frac{\partial \mathcal{L}}{\partial \vb{v}} = \vb{v} - \vb{u} + \lambda  \bm{\phi}(\bar x) = \vb{0} , \qquad \frac{\partial \mathcal{L}}{\partial \lambda} = \bm{\phi}^T(\bar x) \vb{v} - u_D = 0
\end{equation}
If one wanted to solve this in a system form, the following linear system would need to be inverted:
\begin{equation}\label{eq:RODsystem}
\begin{bmatrix} I & \bm{\phi}(\bar x) \\ \bm{\phi}^T(\bar x) & 0 \end{bmatrix} \begin{bmatrix}\vb{v} \\ \lambda \end{bmatrix} = \begin{bmatrix} \vb{u} \\ u_D \end{bmatrix} 
\end{equation}
By solving this system, the polynomial coefficients $\vb{v}$ are obtained and $v_h(\tilde x)$ can be directly computed.

\begin{proposition}
Computing the ROD-E modified boundary condition by inverting the linear system \eqref{eq:RODsystem} and evaluating the polynomial $v_h$ in $\tilde x$ is equivalent to the following polynomial correction:
\begin{equation}\label{eq:rod_final}
    v_h(\tilde x) = [\bm{\phi}^{\text{ROD-E}}(\tilde x)]^T \mathbf{u} + \alpha^{\text{ROD-E}}(\tilde x) u_D , 
\end{equation}
where 
\begin{equation}
    \alpha^{\text{ROD-E}}(\tilde x) = \frac{\bm{\phi}(\tilde{x})^T \bm{\phi}(\bar{x})}{\bm{\phi}(\bar{x})^T \bm{\phi}(\bar{x})} \quad\text{ and }\quad \bm{\phi}^{\text{ROD-E}}(\tilde x) = \bm{\phi}(\tilde x) - \alpha^{\text{ROD-E}}(\tilde x) \bm{\phi}(\bar x).   
\end{equation}
\end{proposition}

\begin{proof}
The solution of this minimization problem can be found starting from the first optimality condition:
\begin{equation}
    \frac{\partial \mathcal{L}}{\partial \vb{v}} = \vb{v} - \vb{u} + \lambda  \bm{\phi}(\bar x) = \vb{0} \quad\Longleftrightarrow\quad \vb{v} = \vb{u} - \lambda  \bm{\phi}(\bar x) .
\end{equation}
By substituting this expression into the second optimality condition, the following equation for $\lambda$ is obtained:
\begin{equation}
    \frac{\partial \mathcal{L}}{\partial \lambda} = \bm{\phi}^T(\bar x) \vb{v} - u_D  = \bm{\phi}^T(\bar x) [\vb{u} - \lambda  \bm{\phi}(\bar x)]  - u_D = 0 ,
\end{equation}
which can be rearranged to give:
\begin{equation}
    \lambda = \frac{\bm{\phi}^T(\bar x) \vb{u} - u_D}{\bm{\phi}^T(\bar x)\bm{\phi}(\bar x)} .
\end{equation}
Finally, the modified degrees of freedom of the ROD-corrected polynomial read:
\begin{equation}
\vb{v} = \vb{u} - \left( \frac{\bm{\phi}(\bar{x})^T \vb{u} - u_D}{\bm{\phi}(\bar{x})^T \bm{\phi}(\bar{x})} \right) \bm{\phi}(\bar{x})
\end{equation}
The modified polynomial $v_h$ is then used to impose the Dirichlet boundary condition at the computational boundary through the numerical flux $\hat u (v_h(\tilde x),u^+_h)= v_h(\tilde x)$, for an upwind flux.

In order to write ROD as a polynomial correction, it is necessary to consider the evaluation of the $v_h$ in the computational boundary point $\tilde x$, which gives
\begin{equation}
    v_h(\tilde{x}) = \bm{\phi}(\tilde{x})^T \vb{v} = \bm{\phi}(\tilde{x})^T \vb{u} - \left( \frac{\bm{\phi}(\bar{x})^T \vb{u} - u_D}{\bm{\phi}(\bar{x})^T \bm{\phi}(\bar{x})} \right) \bm{\phi}(\tilde{x})^T \bm{\phi}(\bar{x}).
\end{equation}
It is now easy to infer that the ROD modified boundary condition can be recast as the polynomial correction \eqref{eq:rod_final}.

\end{proof}
\begin{remark}[Efficiency of ROD polynomial correction]
It should be noticed that by writing ROD-E as in equation \eqref{eq:rod_final}, the reconstructed value $v_h(\tilde x)$ is expressed as a linear combination of the internal degrees of freedom $\vb{u}$ and a set of modified basis functions $\bm{\phi}^{\text{ROD-E}}(\tilde x)$, which only need to be precomputed once, if the boundary is not moving. No inversion of a linear system is needed.
\end{remark}

\begin{remark}[Similarities with the SB polynomial correction]
As a matter of fact, by writing ROD-E as in equation \eqref{eq:rod_final}, it is possible to make a direct link with the SB correction \eqref{eq:SBcorrection}.
Simply, the SB polynomial correction is retrieved if $\alpha^{\text{ROD-E}}(\tilde x) = 1$. 
\end{remark}

\begin{remark}[Consistency with body fitted boundary conditions]
Consistency is also maintained, as for body-fitted boundary conditions, $d=0$ and $\bar x = \tilde x$.
It is easily shown that this brings to $\alpha^{\text{ROD-E}}(\tilde x)=1$ and $\bm{\phi}^{\text{ROD-E}}(\tilde x)=\vb{0}$, which eventually gives $v_h(\tilde x)=u_D$.
\end{remark}

\subsection{$L^2$ norm as cost function}\label{sec:RODL2}

The ROD-$L^2$ method reconstructs a modified polynomial $v_h$ that minimizes the $L^2$ distance with the internal polynomial $u_h$, while satisfying the Dirichlet condition on the real boundary.

In this case, the cost function reads:
\[
J(\vb{v}) =  \mathcal{D}^{L^2}(\mathbf{v}, \mathbf{u}) = \frac{1}{2} \int_{\Omega_b} (v_h(x) - u_h(x))^2 \text{d}x
\]

Following the same procedure as above, the constrained minimization problem can be solved introducing a Lagrange multiplier:
\[
\mathcal{L}(\vb{v}, \lambda) = \frac{1}{2} \int_{\Omega_b} (v_h(x) - u_h(x))^2 \text{d}x + \lambda(v_h(\bar{x}) - u_D)
\]
In this case, the first optimality condition gives:
\[
\frac{\partial \mathcal{L}}{\partial v_k} = \int_{\Omega_b} (v_h(x) - u_h(x)) \phi_k(x) \text{d}x + \lambda \phi_k(\bar{x}) = 0, \quad k = 0,\dots,p
\]

Substituting the polynomial expansions, the mass matrix appears in the first optimality condition. The problem to solve can be rewritten in the following matrix form:
\[
\frac{\partial \mathcal{L}}{\partial \vb{v}} = M(\vb{v} - \vb{u}) + \lambda \bm{\phi}(\bar{x}) = \vb{0},  \qquad \frac{\partial \mathcal{L}}{\partial \lambda} = \bm{\phi}(\bar{x})^T \vb{v} - u_D = 0 .
\]
Again, by solving the following linear system, 
\begin{equation}\label{eq:RODL2system}
\begin{bmatrix} M & \bm{\phi}(\bar x) \\ \bm{\phi}^T(\bar x) & 0 \end{bmatrix} \begin{bmatrix}\vb{v} \\ \lambda \end{bmatrix} = \begin{bmatrix} M \vb{u} \\ u_D \end{bmatrix} 
\end{equation}
the coefficients $\vb{v}$ are obtained, and the evaluated polynomial $v_h(\tilde x)$ can be computed.

\begin{proposition}
Computing the ROD-$L^2$ modified boundary condition by inverting the linear system \eqref{eq:RODL2system} and evaluating the polynomial $v_h$ in $\tilde x$ is equivalent to the following polynomial correction:
\begin{equation}\label{eq:rodL2_final}
    v_h(\tilde x) = [\bm{\phi}^{\text{ROD-$L^2$}}(\tilde x)]^T \mathbf{u} + \alpha^{\text{ROD-$L^2$}}(\tilde x) u_D , 
\end{equation}
where 
\begin{equation}
    \alpha^{\text{ROD-$L^2$}}(\tilde{x}) = \frac{\bm{\phi}(\tilde{x})^T M^{-1} \bm{\phi}(\bar{x})}{\bm{\phi}(\bar{x})^T M^{-1} \bm{\phi}(\bar{x})} \quad\text{ and }\quad \bm{\phi}^{\text{ROD-$L^2$}}(\tilde x) = \bm{\phi}(\tilde x) - \alpha^{\text{ROD-$L^2$}}(\tilde x) \bm{\phi}(\bar x).   
\end{equation}

\end{proposition}

\begin{proof}

From the first optimality condition, we obtain the following relation for $\vb{v}$: 
\[
\vb{v} = \vb{u} - \lambda M^{-1} \bm{\phi}(\bar{x}).
\]
Replacing $\vb{v}$ into the second optimality condition and solving for $\lambda$ gives:
\[
\bm{\phi}(\bar{x})^T [\vb{u} - \lambda M^{-1} \bm{\phi}(\bar{x})] - u_D = 0 \qquad\Longleftrightarrow\qquad \lambda = \frac{\bm{\phi}(\bar{x})^T \vb{u} - u_D}{\bm{\phi}(\bar{x})^T M^{-1} \bm{\phi}(\bar{x})}
\]
In this case, the modified degrees of freedom are:
\[
\vb{v} = \vb{u} - \left( \frac{\bm{\phi}(\bar{x})^T \vb{u} - u_D}{\bm{\phi}(\bar{x})^T M^{-1} \bm{\phi}(\bar{x})} \right) M^{-1} \bm{\phi}(\bar{x})
\]
As it was done above, in order to write the ROD-$L^2$ as a polynomial correction, it is necessary to evaluate $v_h$ at the computational boundary $\tilde{x}$:
\begin{equation}\label{eq:rodL2poly}
\begin{split}
    v_h(\tilde{x}) = \bm{\phi}(\tilde{x})^T \vb{v} = \bm{\phi}(\tilde{x})^T \vb{u} - \left( \frac{\bm{\phi}(\bar{x})^T \vb{u} - u_D}{\bm{\phi}(\bar{x})^T M^{-1} \bm{\phi}(\bar{x})} \right) \bm{\phi}(\tilde{x})^T M^{-1} \bm{\phi}(\bar{x}).
\end{split}
\end{equation}
From equation \eqref{eq:rodL2poly}, it is direct to recast the ROD-$L^2$ method as the polynomial correction \eqref{eq:rodL2_final}. 

\end{proof}

\subsection{General weight matrix in cost function}\label{sec:RODW}

This section aims at providing a general framework that encompasses all possibilities considered above.

The ROD method with a general weight matrix reconstructs a modified polynomial $v_h$ that minimizes a weighted distance with the internal polynomial $u_h$, while satisfying the Dirichlet condition on the real boundary.

The cost function now depends on the general symmetric positive definite matrix $W$:
\[
\mathcal{J}(\mathbf{v}) = \frac{1}{2} (\mathbf{v} - \mathbf{u})^T W (\mathbf{v} - \mathbf{u}).
\]
The constrained minimization problem is solved using the Lagrangian:
\[
\mathcal{L}(\mathbf{v}, \lambda) = \frac{1}{2} (\mathbf{v} - \mathbf{u})^T W (\mathbf{v} - \mathbf{u}) + \lambda(v_h(\bar{x}) - u_D)
\]
Following the same steps as above, the ROD-W corrected boundary value can be expressed as:
\[
    v_h(\tilde x) = [\bm{\phi}^{\text{ROD-W}}(\tilde x)]^T \mathbf{u} + \alpha^{\text{ROD-W}}(\tilde x) u_D , 
\]
where
\[
\alpha^{\text{ROD-W}}(\tilde{x}) = \frac{\bm{\phi}(\tilde{x})^T W^{-1} \bm{\phi}(\bar{x})}{\bm{\phi}(\bar{x})^T W^{-1} \bm{\phi}(\bar{x})} \quad\text{ and }\quad \bm{\phi}^{\text{ROD-W}}(\tilde x) = \bm{\phi}(\tilde x) - \alpha^{\text{ROD-W}}(\tilde x) \bm{\phi}(\bar x).
\]
It is straightforward to show that, when $W = I$ (identity matrix), we recover the ROD-E method.
Instead, when $W = M$ (mass matrix), we recover the ROD-$L^2$ method.

The SB correction can be interpreted as a ROD method with a specific weight matrix $W_{\text{SB}}$ that satisfies:
\[
\frac{\bm{\phi}(\tilde{x})^T W^{-1}_{\text{SB}} \bm{\phi}(\bar{x})}{\bm{\phi}(\bar{x})^T W^{-1}_{\text{SB}} \bm{\phi}(\bar{x})} = 1
\]
The corresponding weight matrix can be computed as:
\[
\bm{\phi}(\tilde{x})^T W^{-1}_{\text{SB}} \bm{\phi}(\bar{x})= \bm{\phi}(\bar{x})^T W^{-1}_{\text{SB}} \bm{\phi}(\bar{x}) \quad \Longleftrightarrow \quad \bm{\delta}^T W^{-1}_{\text{SB}} \bm{\phi}(\bar{x}) = 0
\]
where $\bm{\delta} = \bm{\phi}(\tilde{x}) - \bm{\phi}(\bar{x})$.

Although rather unlikely in actual problems, if $W^{-1}_{\text{SB}}$ would have by chance the following shape:
$$W^{-1}_{\text{SB}} = I - \frac{\bm{\delta}\bm{\delta}^T}{\bm{\delta}^T\bm{\delta}},$$ 
the ROD method would coincide with the SB method.
As a matter of fact, after simple computations, this matrix gives
$$ \bm{\delta}^T \left(I - \frac{\bm{\delta}\bm{\delta}^T}{\bm{\delta}^T\bm{\delta}}\right) \bm{\phi}(\bar{x}) = \bm{\delta}^T \bm{\phi}(\bar{x}) - \bm{\delta}^T \frac{\bm{\delta}\bm{\delta}^T}{\bm{\delta}^T\bm{\delta}} \bm{\phi}(\bar{x}) = 0. $$

\subsection{ROD as a polynomial correction with multiple constraints} \label{ref:ROD2D}

Although the analysis performed in this work is done on one-dimensional domains, the formulation of the ROD method as a polynomial correction can be expressed in a similar way also for multi-dimensional problems, meaning where there are more constraints for each boundary cell. Indeed, when working in multiple dimensions, the boundary term is no longer a single point as in 1D, but it is the surface integral computed on the computational boundary, which is in general defined as the union of the closest mesh edges, of the starting (easily generated) mesh, to the real embedded boundary . 
As it was shown on recent works \cite{santos2024very,ciallella2024very}, where ROD is coupled with DG, in general a single ROD polynomial for each boundary cell is found by introducing several costraints in the ROD minimization problem. Then, this polynomial is used to reconstruct the corrected value in each quadrature point of the edge integral, performed on the computational boundary. 

The following multi-dimensional formulation is presented under the name ROD2D method, and focuses on the Euclidean and $L^2$ norms. 
However, it should be noticed that the same considerations are also valid for 3D problems. 
In this case, the minimization problem can be written as follows.
Let $u_h$ be the polynomial approximation of degree $p$ in the boundary element $\Omega_b$.
A modified polynomial $v_h$ is defined as:
\begin{equation}
    v_h(\vb{x}) = \bm{\phi}^T(\vb{x}) \vb{v},
\end{equation}
where the $(p+1)^2$ coefficients $\mathbf{v} = [v_{0}, \ldots, v_{(p+1)^2}]^T$ are determined by minimizing the distance to $\mathbf{u}$ subject to the boundary constraints $v_h(\bar{\vb{x}}_k)=u_D(\vb{\bar x}_k)$ for $k=1,\ldots,K$,
where $\vb{x}=(x,y)$. 
Let us take $\mathbf{u}_D=[u_D(\bar x_1,\bar y_1),\ldots, u_D(\bar x_K,\bar y_K)]$ where $K$ is the number of constraints in each boundary cell.
Following the same notation as above, $\bar{\vb{x}}$ and $\tilde{\vb{x}}$ refer to points on the real and computational boundary, respectively.

When considering multiple constraints, the functional of the ROD2D-E method becomes:
\begin{equation}
\mathcal{L}(\mathbf{v}, \bm{\lambda}) = \mathcal{D}^{E}(\mathbf{v},\mathbf{u}) + \sum_{k=1}^K \lambda_k \left( v_h(\bar{\vb{x}}_k) - u_D(\bar{\vb{x}}_k) \right)
\end{equation}
where $\bm{\lambda} = [\lambda_1, \ldots, \lambda_K]^T$ is the vector of Lagrange multipliers.
In this case, the optimality conditions gives:
\begin{equation}
\frac{\partial \mathcal{L}}{\partial \vb{v}} = \vb{v} - \vb{u} + \Phi(\bar{\vb{x}}) \bm{\lambda}  = 0, \qquad \frac{\partial \mathcal{L}}{\partial \bm{\lambda}} = \Phi^T(\bar{\vb{x}})\vb{v}  - \vb{u}_D = 0
\end{equation}
where $[\Phi(\bar{\vb{x}})]_{j k} = \phi_j(\vb{\bar x}_k)$.

In earlier works related to the development of ROD for discontinuous Galerkin \cite{santos2024very,ciallella2024very}, the following linear system is solved to obtain the coefficients $\vb{v}$:
\begin{equation}\label{eq:RODE2Dsystem}
\begin{bmatrix} I & \Phi(\bar{\vb{x}}) \\ \Phi^T(\bar{\vb{x}}) & 0 \end{bmatrix} \begin{bmatrix}\vb{v} \\ \bm{\lambda} \end{bmatrix} = \begin{bmatrix} \vb{u} \\ \vb{u}_D \end{bmatrix}, 
\end{equation}
Then,  the coefficients $\vb{v}$ can be used to compute the 2D ROD boundary condition by evaluating it at $\tilde{\vb{x}}$.

\begin{proposition}\label{prop:ROD2DE}
Computing the ROD2D-E modified boundary condition in multiple space dimension by inverting the linear system \eqref{eq:RODE2Dsystem} and evaluating the polynomial $v_h$ in $\tilde{\vb{x}}$ is equivalent to the following polynomial correction:
\begin{equation}\label{eq:rod2DE_final}
v_h(\tilde{\vb{x}}) = [\bm{\phi}^{\text{ROD2D-E}}(\tilde{\vb{x}})]^T \mathbf{u} + [\bm{\alpha}^{\text{ROD2D-E}}]^T  \mathbf{u}_D
\end{equation}
where
\begin{equation}
    \begin{split}
        [\bm{\alpha}^{\text{ROD2D-E}}]^T &= \bm{\phi}^T(\tilde{\vb{x}}) \Phi(\bar{\mathbf{x}}) (\Phi^T(\bar{\vb{x}})\Phi(\bar{\vb{x}}))^{-1}, \\
        [\bm{\phi}^{\text{ROD2D-E}}(\tilde{\vb{x}})] &= \bm{\phi}(\tilde{\vb{x}}) -  \Phi(\bar{\mathbf{x}}) \bm{\alpha}^{\text{ROD2D-E}}.
    \end{split}
\end{equation}

\end{proposition}

\begin{proof}

The optimality conditions provides the following system to solve:
\begin{equation}
\vb{v} = \vb{u} - \Phi(\bar{\vb{x}}) \bm{\lambda}, \qquad  \Phi^T(\bar{\vb{x}})\vb{v}  - \vb{u}_D = 0
\end{equation}
Substituting the expression for $\vb{v}$ into the second optimality condition gives:
\begin{equation}
\Phi^T(\bar{\vb{x}}) \left[\vb{u} - \Phi(\bar{\vb{x}}) \bm{\lambda} \right] - \vb{u}_D =0 \qquad\Longleftrightarrow \qquad \Phi^T(\bar{\vb{x}})\Phi(\bar{\vb{x}}) \bm{\lambda} = \Phi^T(\bar{\vb{x}})\vb{u} - \vb{u}_D
\end{equation}
From the latter, it is possible to compute the Lagrange multipliers:
\begin{equation}
\bm{\lambda} = (\Phi^T(\bar{\vb{x}})\Phi(\bar{\vb{x}}))^{-1} (\Phi^T(\bar{\vb{x}})\vb{u} - \vb{u}_D).
\end{equation}

Finally, the modified polynomial evaluated at the computational boundary $\tilde{\vb{x}}$ can be written as:
\begin{equation}
\begin{split}
    v_h(\tilde{\vb{x}}) &= \bm{\phi}^T(\tilde{\vb{x}}) \vb{v} = \bm{\phi}^T(\tilde{\vb{x}}) \left[ \mathbf{u} - \Phi(\bar{\mathbf{x}}) \bm{\lambda} \right] \\
    & =  \bm{\phi}^T(\tilde{\vb{x}}) \left[ \mathbf{u} - \Phi(\bar{\mathbf{x}})  (\Phi^T(\bar{\vb{x}})\Phi(\bar{\vb{x}}))^{-1}  ( \Phi^T(\bar{\mathbf{x}}) \mathbf{u} - \mathbf{u}_D ) \right] \\
    & = \left[ \bm{\phi}^T(\tilde{\vb{x}}) - \bm{\phi}^T(\tilde{\vb{x}}) \Phi(\bar{\mathbf{x}})  (\Phi^T(\bar{\vb{x}})\Phi(\bar{\vb{x}}))^{-1}  \Phi^T(\bar{\mathbf{x}}) \right] \mathbf{u} +\\ 
    &\qquad\qquad\qquad\qquad\qquad\qquad\qquad\qquad\qquad \bm{\phi}^T(\tilde{\vb{x}}) \Phi(\bar{\mathbf{x}})  (\Phi^T(\bar{\vb{x}})\Phi(\bar{\vb{x}}))^{-1}  \mathbf{u}_D 
\end{split}
\end{equation}

It is straightforward now that the ROD2D-E boundary condition can be recast more elegantely as the polynomial correction \eqref{eq:rod2DE_final}.    
\end{proof}

\begin{remark}
This manipulation allows one to recast the ROD method in multiple space dimensions as a polynomial correction that only needs the inversion of the matrix $\Phi^T(\bar{\vb{x}})\Phi(\bar{\vb{x}})$. This reduces the computational costs usually needed to invert the global system at each time iteration, since this matrix is smaller in dimension. Moreover, it could be computed only once for each element, if the boundary is not moving, at the beginning of the simulation and stored.
If the boundary is moving, $\Phi^T(\bar{\vb{x}})\Phi(\bar{\vb{x}})$ needs to be inverted at each iteration, but it would still be more efficient than inverting the global system \eqref{eq:RODE2Dsystem} since it is much smaller in dimension.
\end{remark}

A similar procedure can be followed to recover the ROD2D-$L^2$ polynomial correction.
In this case the functional is
\begin{equation}
\mathcal{L}(\vb{v}, \bm{\lambda}) = \frac{1}{2} \int_{\Omega_b} (v_h(\vb{x}) - u_h(\vb{x}))^2 \text{d}\vb{x} + \sum_{k=1}^m \lambda_k (v_h(\bar{\vb{x}}_k) - u_D(\bar{\vb{x}}_k)).
\end{equation}
Again, one could directly solve the following linear system to obtain the coefficients $\vb{v}$,
\begin{equation}\label{eq:RODL2-2Dsystem}
\begin{bmatrix} M & \Phi(\bar{\vb{x}}) \\ \Phi^T(\bar{\vb{x}}) & 0 \end{bmatrix} \begin{bmatrix}\vb{v} \\ \bm{\lambda} \end{bmatrix} = \begin{bmatrix} M \vb{u} \\ \vb{u}_D \end{bmatrix} ,
\end{equation}
and then evaluate the polynomial $v_h$ on the computational boundary $\tilde{\vb{x}}$.

\begin{proposition}
Computing the ROD2D-$L^2$ modified boundary condition in multiple space dimension by inverting the linear system \eqref{eq:RODL2-2Dsystem} and evaluating the polynomial $v_h$ in $\tilde{\vb{x}}$ is equivalent to the following polynomial correction:
\begin{equation} \label{eq:rod2DL2_final}
v_h(\tilde{x}) = [\bm{\phi}^{\text{ROD2D-$L^2$}}(\tilde{\vb{x}})]^T \mathbf{u} + [\bm{\alpha}^{\text{ROD2D-$L^2$}}]^T \mathbf{u}_D
\end{equation}
where
\begin{equation}
    \begin{split}
        [\bm{\alpha}^{\text{ROD2D-$L^2$}}]^T &= \bm{\phi}^T(\tilde{\vb{x}})M^{-1} \Phi(\bar{\mathbf{x}}) (\Phi^T(\bar{\vb{x}})M^{-1}\Phi(\bar{\vb{x}}))^{-1}, \\ 
        [\bm{\phi}^{\text{ROD2D-$L^2$}}(\tilde{\vb{x}})] &= \bm{\phi}(\tilde{\vb{x}}) -  \Phi(\bar{\mathbf{x}}) \bm{\alpha}^{\text{ROD2D-$L^2$}}.
    \end{split}
\end{equation}

\end{proposition}

\begin{proof}

The optimality conditions provides the following system to solve:
\begin{equation}
\vb{v} = \vb{u} - M^{-1}\Phi(\bar{\vb{x}}) \bm{\lambda}, \qquad  \Phi^T(\bar{\vb{x}})\vb{v}  - \vb{u}_D = 0
\end{equation}
Substituting the expression for $\vb{v}$ into the second optimality condition gives:
\begin{equation}
\Phi^T(\bar{\vb{x}}) \left[\vb{u} - M^{-1}\Phi(\bar{\vb{x}}) \bm{\lambda} \right] - \vb{u}_D =0 \qquad\Longleftrightarrow \qquad \Phi^T(\bar{\vb{x}}) M^{-1} \Phi(\bar{\vb{x}}) \bm{\lambda} = \Phi^T(\bar{\vb{x}})\vb{u} - \vb{u}_D
\end{equation}
From the latter, it is possible to compute the Lagrange multipliers:
\begin{equation}
\bm{\lambda} = (\Phi^T(\bar{\vb{x}}) M^{-1} \Phi(\bar{\vb{x}}))^{-1} (\Phi^T(\bar{\vb{x}})\vb{u} - \vb{u}_D).
\end{equation}

Finally, the modified polynomial evaluated at the computational boundary $\tilde{\vb{x}}$ can be written as:
\begin{equation}
\begin{split}
    v_h(\tilde{\vb{x}}) &= \bm{\phi}^T(\tilde{\vb{x}}) \vb{v} = \bm{\phi}^T(\tilde{\vb{x}}) \left[ \mathbf{u} - M^{-1}\Phi(\bar{\mathbf{x}}) \bm{\lambda} \right] \\
    & =  \bm{\phi}^T(\tilde{\vb{x}}) \left[ \mathbf{u} - M^{-1}\Phi(\bar{\mathbf{x}})  (\Phi^T(\bar{\vb{x}})M^{-1}\Phi(\bar{\vb{x}}))^{-1}  ( \Phi^T(\bar{\mathbf{x}}) \mathbf{u} - \mathbf{u}_D ) \right] \\
    & = \left[ \bm{\phi}^T(\tilde{\vb{x}}) - \bm{\phi}^T(\tilde{\vb{x}}) M^{-1}\Phi(\bar{\mathbf{x}})  (\Phi^T(\bar{\vb{x}}) M^{-1} \Phi(\bar{\vb{x}}))^{-1}  \Phi^T(\bar{\mathbf{x}}) \right] \mathbf{u} + \\
    &\qquad\qquad\qquad\qquad\qquad\qquad \bm{\phi}^T(\tilde{\vb{x}}) M^{-1}\Phi(\bar{\mathbf{x}})  (\Phi^T(\bar{\vb{x}})M^{-1}\Phi(\bar{\vb{x}}))^{-1}  \mathbf{u}_D 
\end{split}
\end{equation}

It is straightforward now that the ROD2D-$L^2$ boundary condition can be recast as the polynomial correction \eqref{eq:rod2DL2_final}. 
\end{proof}

\begin{remark}[ROD 1D correction for multi-D problems]
In this section, we have focused on the multi-constraints ROD approach developed in previous works on ROD-DG methods \cite{santos2024very,ciallella2024very}. However, it should be noticed that by writing the 1D ROD method as the polynomial corrections \eqref{eq:rod_final} and \eqref{eq:rodL2_final}, the ROD method in multiple dimensions could be also developed by employing several of the 1D polynomial corrections applied locally in each quadrature point. 
On one side, this would keep the minimization-based approach of the ROD method also in multiple dimension without the need of including multiple constraints in the reconstruction. On the other, it simplifies even more the implementation and the coding effort needed to integrate it in a general computational framework, making it as efficient as the SB polynomial correction. Future works will be dedicated to test this alternative formulations in multiple space dimensions.        
\end{remark}

\section{Stability analysis}\label{sec:stability}

The eigenspectrum analysis is performed with a minimal system configuration. 
We consider two elements containing $p+1$ modal coefficients each.
Due to Dirichlet boundary conditions the analysis cannot consider an infinite domain, as done for the standard Von Neumann analysis, and will be dependent on the number of elements considered in the global system. 
However, as it was proven in \cite{ciallella2025stability} for a similar configuration, the qualitative behavior of the results remains largely unaffected by the specific number of cells.
This is also typical of other analysis involving boundary conditions that have been presented in the literature \cite{M2AN_2015__49_1_39_0}.
As a matter of fact, the analysis in \cite{ciallella2025stability} of the SB method showed that only the eigenvalues of the boundary cell are affected by the embedded boundary treatment.
To simplify the visualization, for the analysis involving the embedded treatments, the reference maximum CFL number is taken as that of the internal DG scheme, computed assuming a periodic domain. 
The analysis is performed taking $\Delta x =1$.

In section \ref{sec:periodic}, the analysis is performed by considering periodic boundary conditions, and, as expected, this allows us to recover the classical CFL stability constraints of DG schemes. 
In section \ref{sec:RODexplicit}, resp.\ section \ref{sec:RODimplicit}, the analysis of the ROD polynomial correction with explicit, resp.\ implicit, time integration is presented. 

Implicit Euler as time integration is also studied to show whether this can have a positive impact on the stability regions of the schemes, as it was observed in \cite{ciallella2025stability} for the SB method. 

\subsection{Periodic boundary conditions and explicit time integration}\label{sec:periodic}

The DG method, following the notation of equation \eqref{eq:discontinuous_galerkin_matrix_form}, for a system of two elements with periodic boundary conditions imposed at the left interface of the first cell $\Omega_1=[1/2,3/2]$, and at the right interface of the second cell $\Omega_2=[3/2,5/2]$, reads
\begin{equation}\label{eq:system_periodic}
    \begin{split}
	M_1 \frac{\text{d} \mathbf{u}_1}{\text{d} t} &= K^s_1 \mathbf{u}_1 - ( K^R_1 \mathbf{u}_1 - K^L_1 \mathbf{u}_2 ), \\
	M_2 \frac{\text{d} \mathbf{u}_2}{\text{d} t} &= K^s_2 \mathbf{u}_2 - ( K^R_2 \mathbf{u}_2 - K^L_2 \mathbf{u}_1 ).
    \end{split}
\end{equation}
In this case, the global matrices are
$$\mathbf{M} = \begin{bmatrix} M_1 &  0  \\ 0  & M_2 \end{bmatrix}, \quad \text{and}\quad \mathbf{K} = \begin{bmatrix} K^s_1 - K_1^R & K_1^L \\ K_2^L & K^s_2 - K_2^R \end{bmatrix}.$$
Under the assumption of a uniform mesh, the local matrices are identical ($M_1=M_2$ and $K^s_1=K^s_2$) as defined in \eqref{eq:mass_stiff_mat}. 
The interface terms $K^R_1$ and $K^R_2$ follow from \eqref{eq:interface_mat}, and more precisely are given by 
$$K^R_1 =  \phi_i^1(x_{3/2}) \phi_j^1(x_{3/2})  ,\qquad K^L_1 = \phi_i^1(x_{1/2}) \phi_j^{2}(x_{5/2}),$$
and  
$$K^R_2 =  \phi_i^2(x_{5/2}) \phi_j^2(x_{5/2})  ,\qquad K^L_2 = \phi_i^2(x_{3/2}) \phi_j^{1}(x_{3/2}),$$
with $K^L_1$ and $K^L_2$ providing the periodic coupling between cells 1-2 and 2-1, respectively.
The global vector of modal coefficients is $\mathbf{U}(t)=[ \mathbf{u}_1(t) \;  \mathbf{u}_2(t)]^T$. 

As expected, when taking periodic conditions, the standard CFL condition for DG methods is obtained: $\text{CFL}^p_{\text{max}}\approx \frac{1}{2p+1}$. 
For completeness, we show in figure \ref{fig:DG_periodic} the maximum amplification factor for this configuration.
In the following sections, to simplify the visualization of stability regions, a normalized $\text{CFL}\in[0,1]$ is defined by scaling the actual CFL value with the periodic-case CFL such that $\frac{\Delta t}{\Delta x} = \text{CFL}^p_{\text{max}} \cdot \text{CFL}$.

\begin{figure}
\centering
\includegraphics[width=0.6\textwidth]{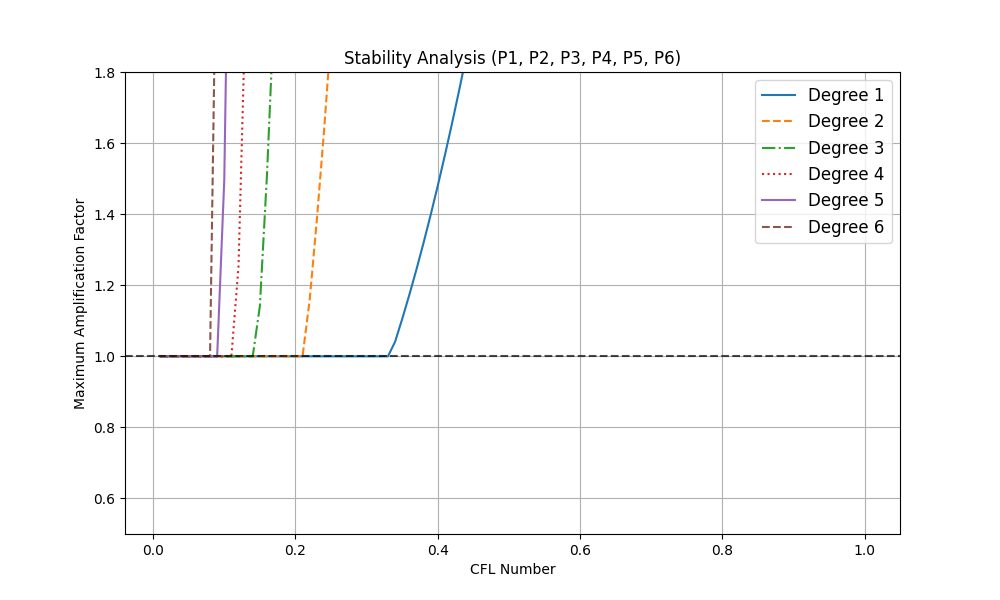}
\caption{Amplification factor: periodic boundary conditions for the explicit discontinuous Galerkin method. When the amplification factor becomes greater than 1, the scheme is unstable.}
\label{fig:DG_periodic}
\end{figure}

\subsection{ROD treatment and explicit time integration}\label{sec:RODexplicit}

The analysis for the ROD boundary treatment can be performed thanks to the reformulations carried out in section \ref{sec:RODcorrection}, which allows one to write easily ROD in matrix form.

In this case, the global system for two cells is
\begin{equation}
\begin{split}
	M_1 \frac{\text{d} \mathbf{u}_1}{\text{d} t} &= K^s_1 \mathbf{u}_1 - ( K^R_1 \mathbf{u}_1 - K^{\text{ROD}}_1 \mathbf{u}_1 ), \\
	M_2 \frac{\text{d} \mathbf{u}_2}{\text{d} t} &= K^s_2 \mathbf{u}_2 - ( K^R_2 \mathbf{u}_2 - K^L_2 \mathbf{u}_1 ),
\end{split}
\end{equation}
where the left interface matrix that introduces the ROD polynomial correction is defined as
$$ K^{\text{ROD}}_1 = \phi^1_i(x_{1/2}) \phi^{\text{ROD},1}_i(x_{1/2}) = \phi^1_i(x_{1/2}) ( \phi^1_j(x_{1/2}) - \alpha^{\text{ROD}}\, \phi^1_j(x_{1/2}+d) ). $$
By varying $\alpha^{\text{ROD}}$, different ROD approaches can be analyzed. 

For the Euclidean distance used for the ROD minimization problem,
$$\alpha^{\text{ROD-E}} =  \frac{\sum_{i=0}^p \phi^1_i(x_{1/2}) \phi^1_i(x_{1/2}+d)}{ \sum_{i=0}^p \phi^1_i(x_{1/2}+d) \phi^1_i(x_{1/2}+d)}, $$
and the stability regions are given in figure \ref{fig:DG_RODE_explicit}.

The ROD-E method with polynomial degree $p=1,2,3$ shows stable areas for all values of $d\in[-1,0]$ under the standard CFL condition, an improvement over the SB method results in \cite{ciallella2025stability}. Instead, for $d\in[0,1]$, the method becomes unstable after a certain threshold $d>d_{max}$, which depends on the polynomial degree used to build the scheme, with the same $d_{max}$ observed for the SB method.
Although the ROD-E method shows better stability properties for $p=1,2,3$ and distance $d\in[-1,0]$ than the SB method, for higher order polynomials $p=4,5,6$, also the ROD-E method becomes unconditionally unstable after a certain threshold.
Contrary to the SB method, the ROD-E method never shows a constraint on the CFL number, it is either stable under a classical CFL condition or it becomes unconditionally unstable after a certain distance threshold.

For the $\mathbb{P}^1$ method, we can analytically compute the eigenvalues of the ROD-E boundary discretized operator, which read
$$
\lambda^{\text{ROD-E}}_{1,2} = \frac{- 3 d^{2} + 5 d  - 2 \pm \sqrt{9 d^{4} - 18 d^{3} + 13 d^{2} - 2 d - 2}}{2 d^{2} - 2 d + 1}.
$$ 
The scheme is stable in the semi-discrete sense when $d<2/3\approx 0.66$, which is the distance threshold after which the scheme becomes unconditionally unstable.

\begin{figure}
\centering
\includegraphics[width=0.98\textwidth,trim={1.5cm 1.5cm 1.5cm 1.5cm},clip]{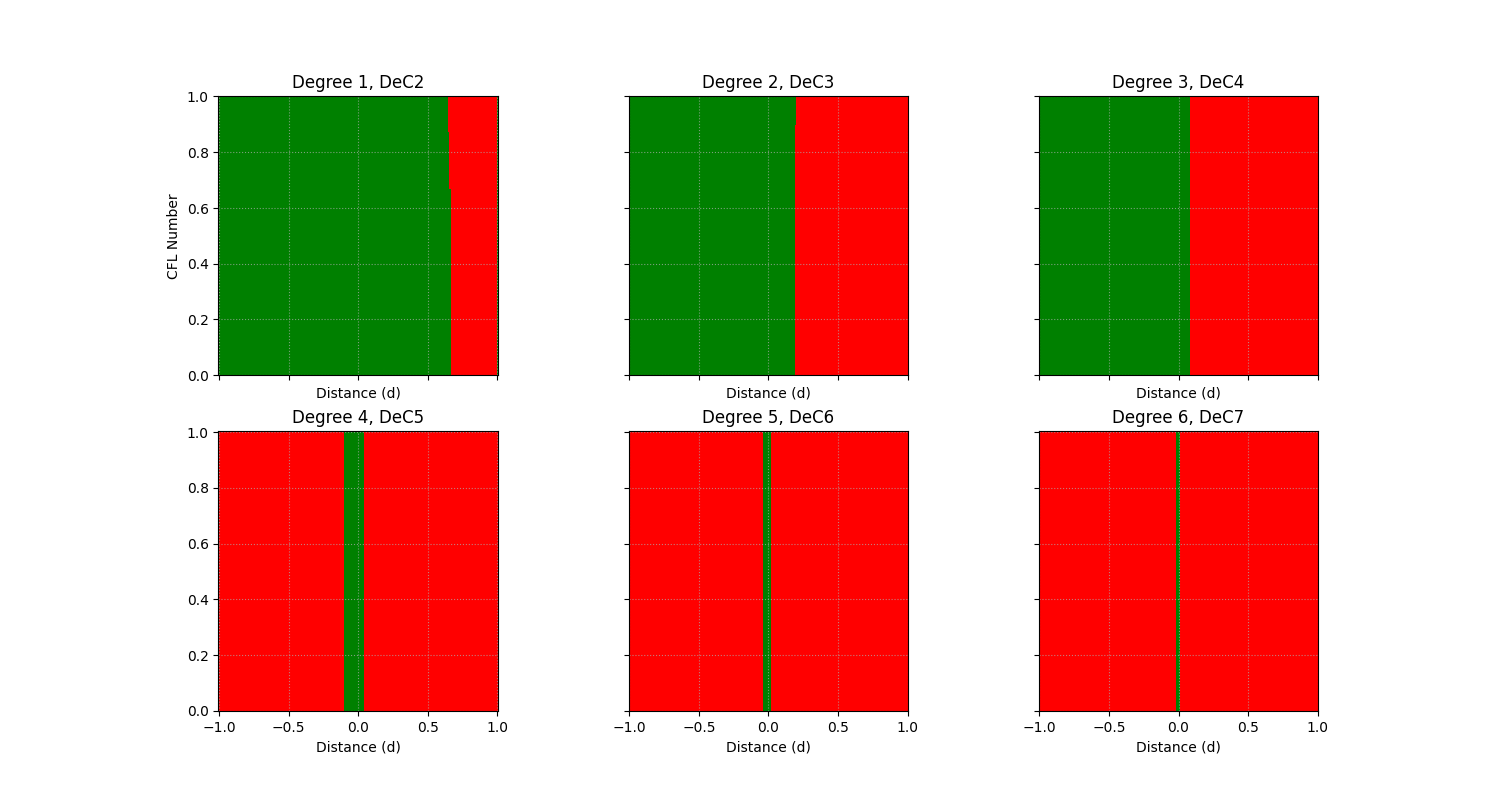}
\caption{Stability region: ROD-E correction of the homogeneous Dirichlet condition for the explicit discontinuous Galerkin method. Green areas are stable, while red areas are unstable.}
\label{fig:DG_RODE_explicit}
\end{figure}

When taking the $L^2$ distance for the ROD minimization problem, 
$$\alpha^{\text{ROD-$L^2$}} =  \frac{\sum_{j=0}^p \sum_{i=0}^p \phi^1_i(x_{1/2}) M^{-1}_{ij} \phi^1_j(x_{1/2}+d)}{\sum_{j=0}^p \sum_{i=0}^p \phi^1_i(x_{1/2}+d) M^{-1}_{ij} \phi^1_j(x_{1/2}+d)}, $$
and the stability regions are given in figure \ref{fig:DG_RODL2_explicit}.

Contrary to the ROD-E method, the ROD-$L^2$ shows stable areas for all values of $d\in[-1,0]$, with no constraint on the CFL with $p=1,2,3,4$, while for $p=5,6$ a constraint on the distance threshold appears. 
Instead, for $d\in[0,1]$, the method becomes unstable after a certain threshold $d>d_{max}$, which depends on the polynomial degree used to build the scheme, as also observed for ROD-E.
It is interesting to notice that by changing the distance function used in the ROD minimization problem, the stability region can be improved. As a matter of fact, for all values of $d\in[-1,0]$, ROD-E was stable only for $p=1,2,3$, while ROD-$L^2$ is also stable for $p=4$.  

The analytical eigenvalues of the ROD-$L^2$ $\mathbb{P}^1$ method are
$$
\lambda^{\text{ROD-$L^2$}}_{1,2} = \frac{- 9 d^{2} + 12 d  - 4 \pm \sqrt{81 d^{4} - 108 d^{3} + 36 d^{2} + 12 d - 8}}{2 \left(3 d^{2} - 3 d + 1\right)}
$$ 
Again, the scheme is stable in the semi-discrete sense when $d<2/3\approx 0.66$.

\begin{figure}
\centering
\includegraphics[width=0.98\textwidth,trim={1.5cm 1.5cm 1.5cm 1.5cm},clip]{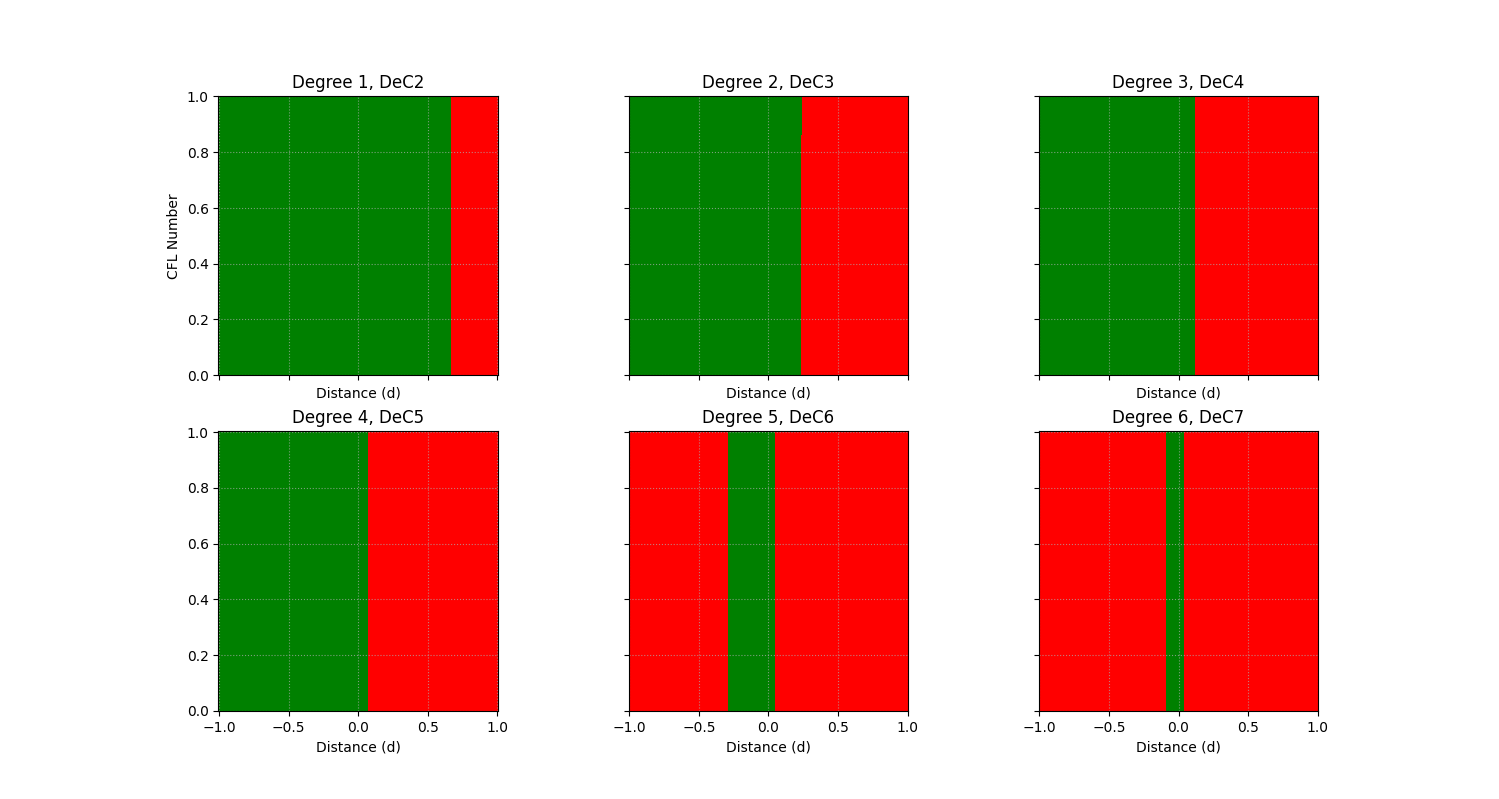}
\caption{Stability region: ROD-$L^2$ correction of the homogeneous Dirichlet condition for the explicit discontinuous Galerkin method. Green areas are stable, while red areas are unstable.}
\label{fig:DG_RODL2_explicit}
\end{figure}

\subsection{ROD treatment and implicit time integration}\label{sec:RODimplicit}

If we consider the implicit Euler time integration of the discontinuous Galerkin method with the ROD (either ROD-E or ROD-$L^2$) polynomial correction,
the global system reads

\begin{equation}
\begin{split}
M_1\frac{\mathbf{u}^{n+1}_1 - \mathbf{u}^{n}_1}{\Delta t} &= K^s_1 \mathbf{u}^{n+1}_1 - ( K^R_1 \mathbf{u}^{n+1}_1 - K^{\text{ROD}}_1 \mathbf{u}^{n+1}_1),\\
M_2\frac{\mathbf{u}^{n+1}_2 - \mathbf{u}^{n}_2}{\Delta t} &= K^s_2 \mathbf{u}^{n+1}_2 - ( K^R_2 \mathbf{u}^{n+1}_2 - K^L_2 \mathbf{u}^{n+1}_1 ),
\end{split}
\end{equation}

In figure \ref{fig:DG_RODE_implicit1}, the stability regions for ROD-E are shown considering $\text{CFL}\in[0,1]$.
Since, for values of $d\in[-1,0]$, the ROD-E with $p=1,2,3$ polynomials and explicit time integration is stable under standard CFL condition, we are interested more about the impact that the implicit time integration may have on higher order methods, i.e.\ with $p=4,5,6$. As for the SB method, for values of $d\in[0,1]$, there is no hope of having an unconditionally stable method.
Even with implicit time integration, the ROD-E method with $p=4,5,6$ is stable only for a value of CFL greater than a certain threshold, which becomes higher when using higher polynomials. In particular, we need $\text{CFL}\geq 3$ for $\mathbb{P}^4$, $\text{CFL}\geq 6$ for $\mathbb{P}^5$, and $\text{CFL}\geq 9$ for $\mathbb{P}^6$.    

\begin{figure}
\centering
\includegraphics[width=0.98\textwidth,trim={1.5cm 1.5cm 1.5cm 1.5cm},clip]{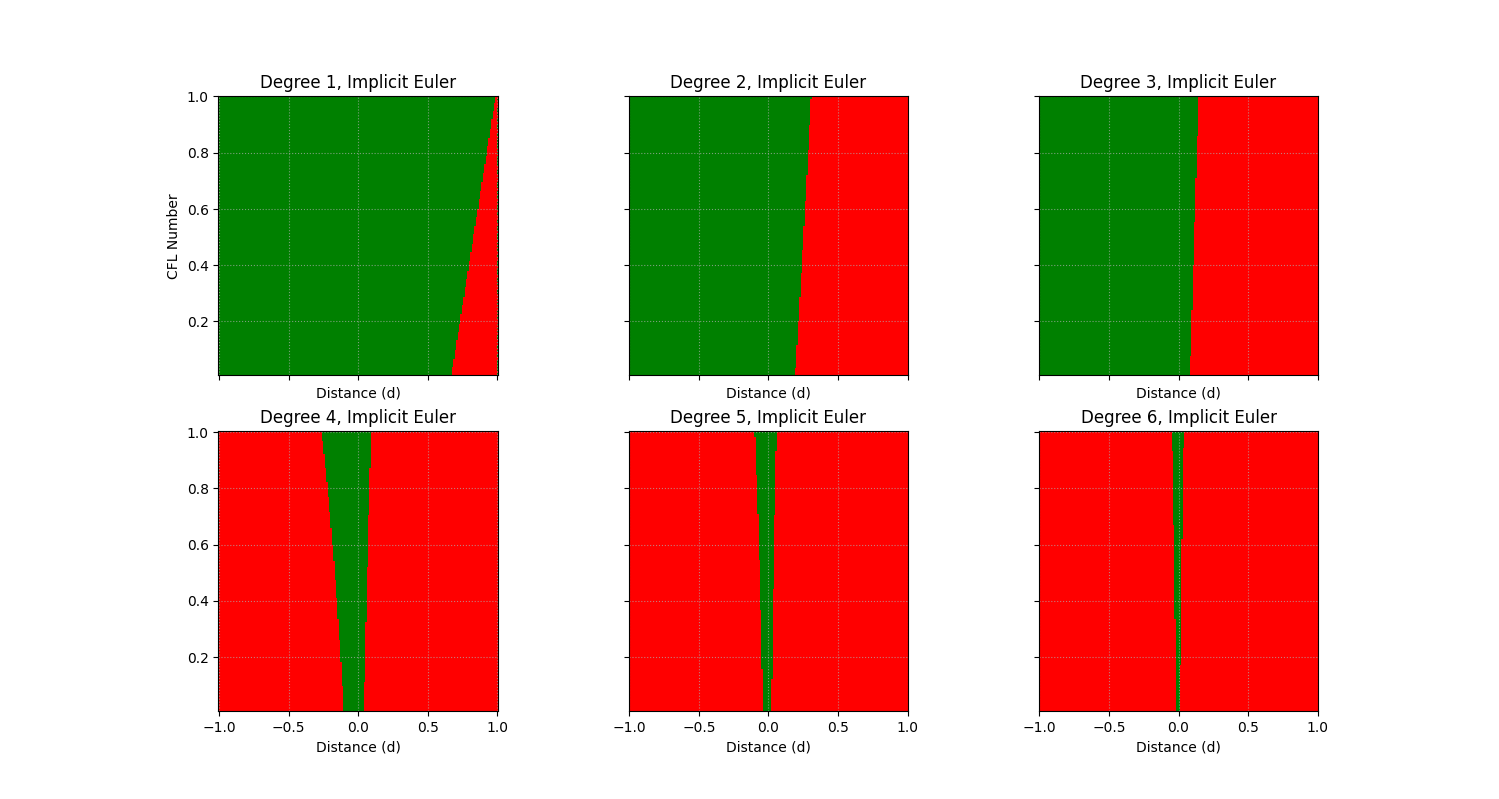}
\caption{Stability region: ROD-E correction of the homogeneous Dirichlet condition for the implicit discontinuous Galerkin method and $\text{CFL}\in[0,1]$. Green areas are stable, while red areas are unstable.}
\label{fig:DG_RODE_implicit1}
\end{figure}

\begin{figure}
\centering
\includegraphics[width=0.98\textwidth,trim={1.5cm 1.5cm 1.5cm 1.5cm},clip]{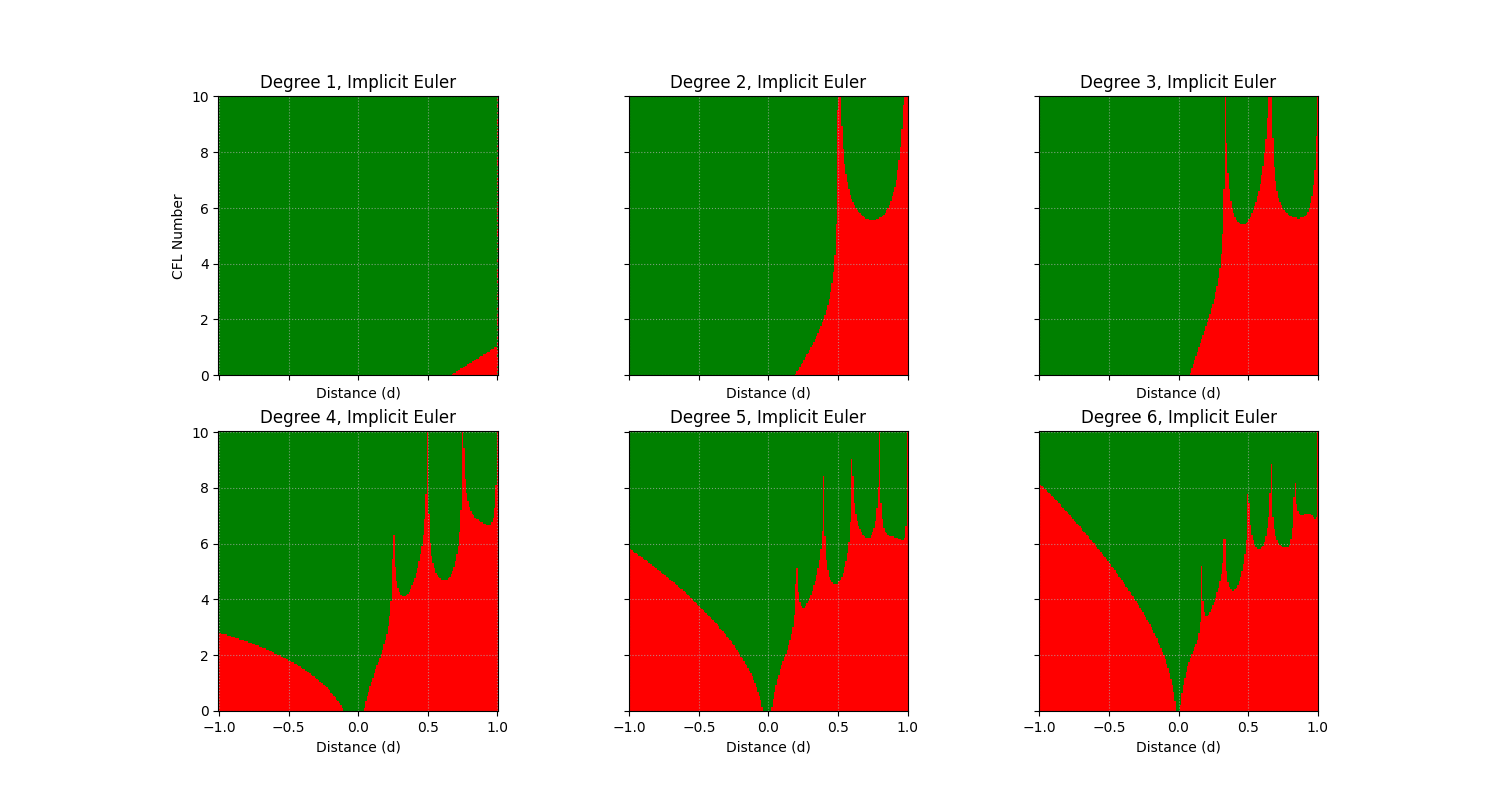}
\caption{Stability region: ROD-E correction of the homogeneous Dirichlet condition for the implicit discontinuous Galerkin method and $\text{CFL}\in[0,10]$. Green areas are stable, while red areas are unstable.}
\label{fig:DG_RODE_implicit10}
\end{figure}

Similarly, also for the ROD-$L^2$ method with implicit time integration, the main focus of the analysis is to study whether it is possible to stabilize the method for very high order polynomials. As a matter of fact, the ROD-$L^2$ method is already stable for all values of $d\in[-1,0]$ and $p=1,2,3,4$, with explicit time integration. 
However, as shown in figure \ref{fig:DG_RODL2_implicit1}, even implicit time integration does not provide a unconditionally stable scheme for all values of $\text{CFL}\in[0,1]$. On the contrary, as shown in figure \ref{fig:DG_RODL2_implicit10}, for $p=5,6$ a minimum CFL is always needed to have a stable method. Although better than ROD-E, we still need $\text{CFL}\geq 0.7$ for $\mathbb{P}^5$, and $\text{CFL}\geq 2$ for $\mathbb{P}^6$ to have a stable scheme.  

\begin{figure}
\centering
\includegraphics[width=0.98\textwidth,trim={1.5cm 1.5cm 1.5cm 1.5cm},clip]{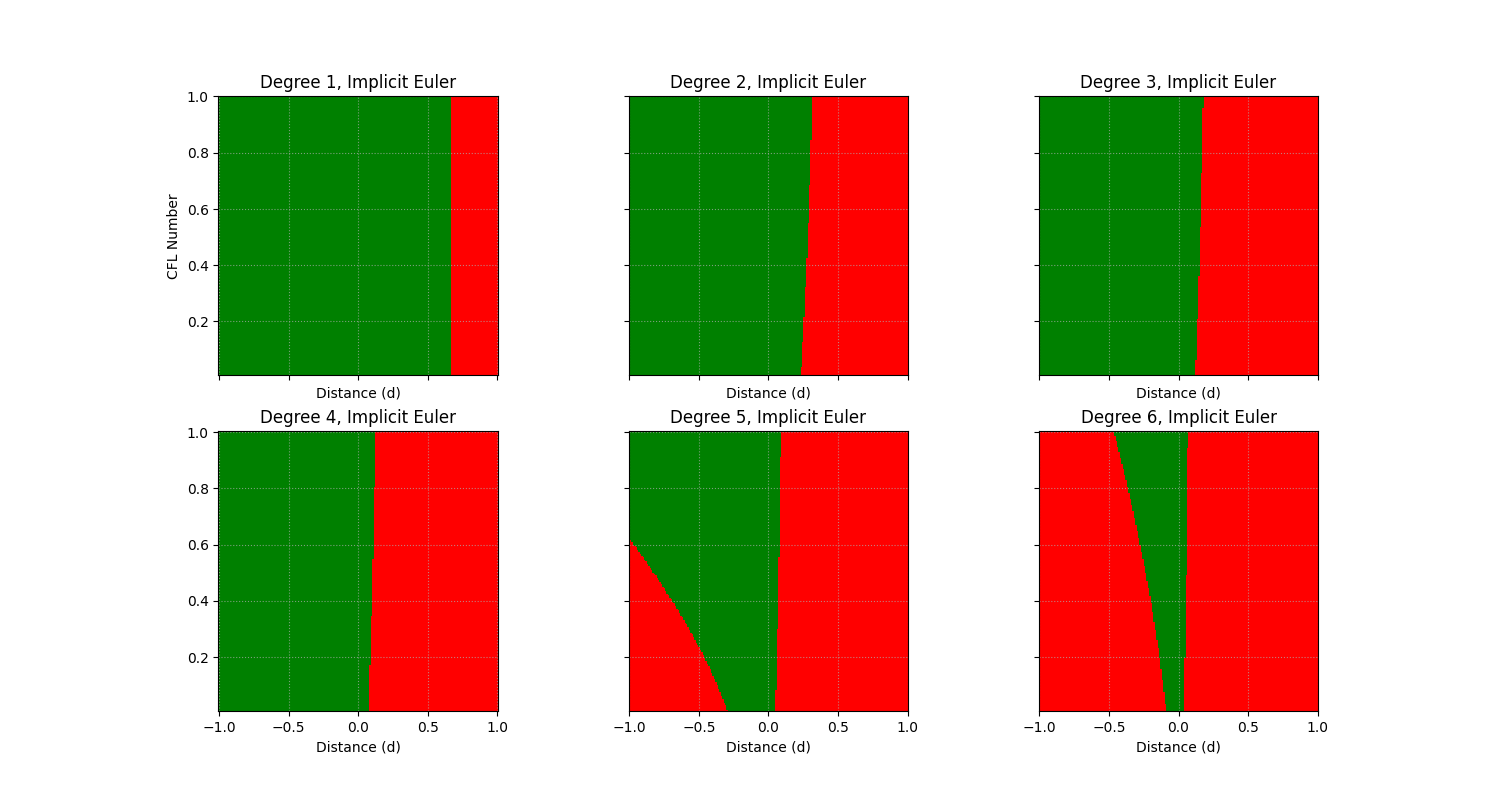}
\caption{Stability region: ROD-$L^2$ correction of the homogeneous Dirichlet condition for the implicit discontinuous Galerkin method and $\text{CFL}\in[0,1]$. Green areas are stable, while red areas are unstable.}
\label{fig:DG_RODL2_implicit1}
\end{figure}

\begin{figure}
\centering
\includegraphics[width=0.98\textwidth,trim={1.5cm 1.5cm 1.5cm 1.5cm},clip]{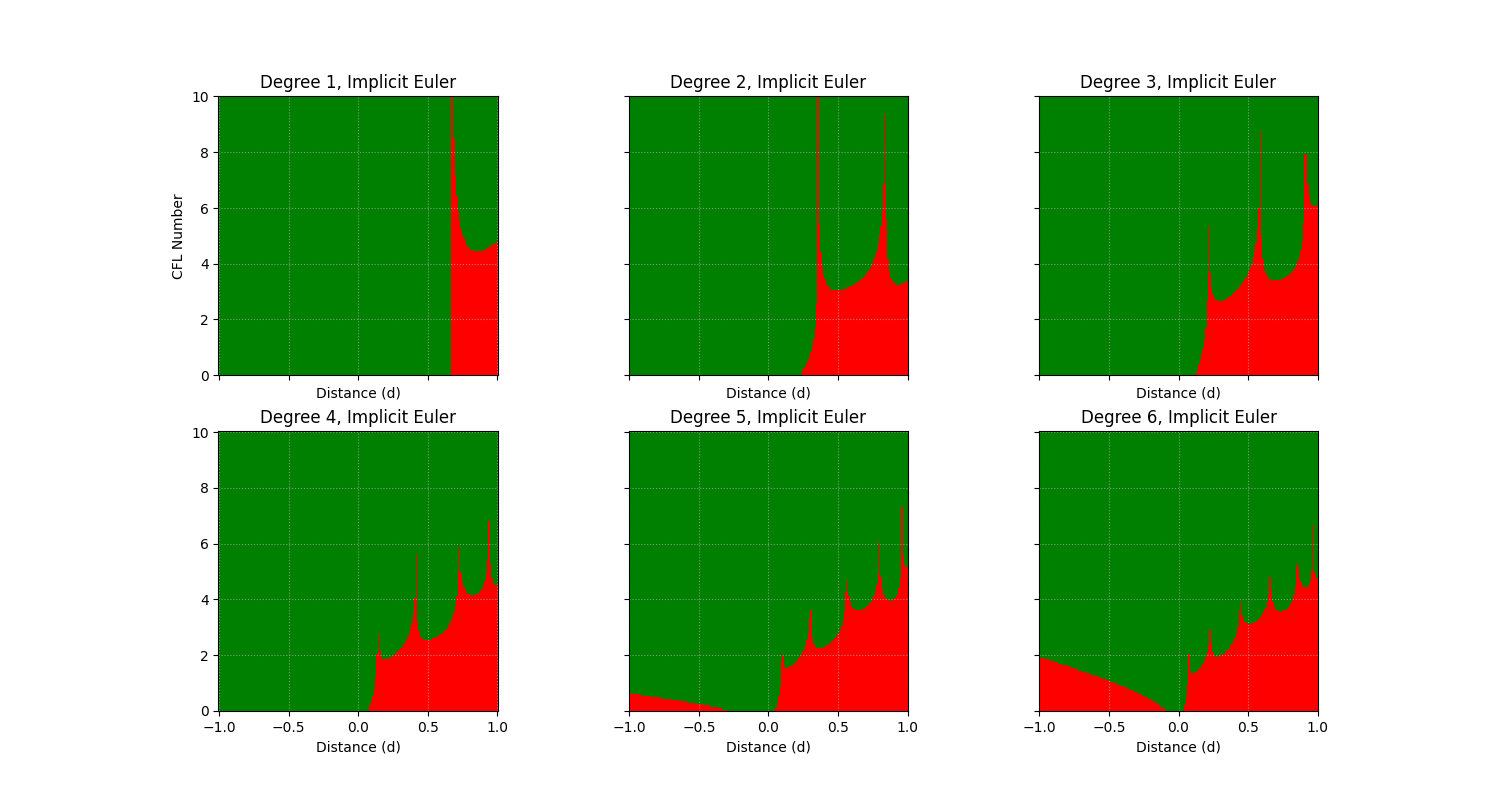}
\caption{Stability region: ROD-$L^2$ correction of the homogeneous Dirichlet condition for the implicit discontinuous Galerkin method and $\text{CFL}\in[0,10]$. Green areas are stable, while red areas are unstable.}
\label{fig:DG_RODL2_implicit10}
\end{figure}

\section{Numerical experiments}\label{sec:tests}

For the numerical experiments, we consider equation \eqref{eq:linear_advection_source} with a stationary manufactured solution 
$$ u_{ex}(x) = 0.1 \sin(\pi x).$$
$u_{ex}(x)$ is taken as initial condition for the numerical simulation.
This makes the solution independent from the accuracy of time integration.
The source term can be directly computed as
$$\partial_t u_{ex}(x) + \partial_x u_{ex}(x) = s(x), \quad\text{resulting in}\quad s(x) = 0.1 \pi \cos(\pi x) .$$
The one-dimensional domain is taken $\Omega=[0,2]$ with a Dirichlet boundary condition imposed at $x_L$.
Since we are dealing with linear advection equation with positive speed, nothing is imposed at $x_R$.
A set of uniform meshes with $N_e=20,40,80,160$ elements is considered for $\mathbb{P}^p$ with $p=1,2,3$. 
While another set of uniform meshes with $N_e=5,10,20,40$ elements is considered for $\mathbb{P}^p$ with $p=4,5,6$.

The ROD (both ROD-E and ROD-$L^2$) polynomial corrections are imposed considering different distance values $d$ of the boundary from the left interface of the first cell. We refer to the previous analysis of the SB method in \cite{ciallella2025stability} to highlight the differences in stability properties between these two families of methods. 
When $d=0$, the boundary coincides with the left interface of the first cell, and therefore
the correction is not applied. This is the case of standard fitted boundary conditions.
Since all studied methods exhibit more instability when $d\in[0,\Delta x]$, the distances considered in the numerical experiments will be taken as $d\in[-\Delta x,0]$, meaning with the real boundary point located outside the first mesh element.
For simplicity, in convergence tables, the distance is reported in non-dimensional form, i.e.\ $d/\Delta x$.
Stability results obtained through the analysis are perfectly corroborated by the numerical results for all configurations of $d$, and for both explicit and implicit time integration.

In table \ref{tab:LAE_RODEexplicit}, we present the numerical results obtained with the ROD-E correction and explicit time integration. Again, we consider the greater distance that makes the scheme stable under the associated maximum CFL value.    
As also mentioned above, for $d<0$, it is possible to reach the maximum distance of $d=-1$ for polynomials $p=1,2,3$ with no condition on the CFL.
For $p=4,5,6$, the scheme is unconditionally unstable for $d=-1$, and and the maximum stable distance is much lower: $d=-0.10$ for $\mathbb{P}^4$,  $d=-0.04$ for $\mathbb{P}^5$, and $d=-0.015$ for $\mathbb{P}^6$.
Contrary to the SB method \cite{ciallella2025stability}, there are no cases where it is possible to stabilize the ROD-E method by lowering the CFL. ROD-E is either stable under a classical CFL condition or unconditionally unstable; this is what happens for higher order polynomials after a certain distance threshold.

\begin{table}
        \caption{Linear advection equation: $L_2$ errors $\|u-u_{ex}\|_2$ and estimated order of accuracy (EOA) with the ROD-E correction and explicit time integration for polynomials $\mathbb{P}^p$.
        CFL and $d$ are chosen to consider the worst case scenario, i.e.\ the furthest possible $d\in[-1,0]$ and the associated maximum CFL.
        }\label{tab:LAE_RODEexplicit}
        \scriptsize
        \centering
        \begin{tabular}{c|cc|cc|cc} \hline
                        &\multicolumn{2}{c|}{$\mathbb{P}^1$, CFL=1.0, $d=-1.00$} &\multicolumn{2}{c|}{$\mathbb{P}^2$, CFL=1.0, $d=-1.00$} &\multicolumn{2}{c}{$\mathbb{P}^3$, CFL=1.0, $d=-1.00$}  \\[0.5mm]
                        \cline{2-7}
                        $N_e$,  & $L_2$ error & EOA  & $L_2$  error & EOA   & $L_2$ error & EOA    \\ \hline 
                        20      &  5.08E-04  &   --  & 1.45E-03  &  --  & 2.28E-05  &  --  \\
                        40      &  1.21E-04  & 2.07  & 1.81E-04  & 2.99 & 7.16E-07  &  4.99  \\
                        80      &  2.98E-05  & 2.01  & 2.26E-05  & 2.99 & 2.23E-08  &  5.00  \\
                        160     &  7.43E-06  & 2.00  & 2.83E-06  & 2.99 & 6.28E-10  &  5.14  \\
        \end{tabular}
        \begin{tabular}{c|cc|cc|cc} \hline
                        &\multicolumn{2}{c|}{$\mathbb{P}^4$, CFL=1.0, $d=-0.10$} &\multicolumn{2}{c|}{$\mathbb{P}^5$, CFL=1.0, $d=-0.04$} &\multicolumn{2}{c}{$\mathbb{P}^6$, CFL=1.0, $d=-0.015$}  \\[0.5mm]
                        \cline{2-7}
                        $N_e$,  & $L_2$ error & EOA  & $L_2$  error & EOA   & $L_2$ error & EOA    \\ \hline 
                        5      &  4.43E-05  & --   &  6.72E-07   &  --  & 4.88E-08  &  --  \\
                        10     &  1.73E-06  & 4.67 &  6.07E-09   & 6.78 & 4.55E-10  &  6.74 \\
                        20     &  5.69E-08  & 4.92 &  6.11E-11   & 6.63 & 3.70E-12  &  6.94 \\
                        40     &  1.80E-09  & 4.98 &  7.59E-13   & 6.33 & 7.47E-15  &  8.95 \\ \hline
        \end{tabular}
\end{table}

In table \ref{tab:LAE_RODL2explicit}, we present the numerical results obtained with the ROD-$L^2$ correction and explicit time integration. As before, we consider the greater distance that makes the scheme stable under the associated maximum CFL value.    
As also mentioned above, for $d<0$, it is possible reach the maximum distance of $d=-1$ for polynomials $p=1,2,3,4$ with no condition on the CFL, meaning the $\mathbb{P}^4$ method is stable, unlike the ROD-E case.
For $p=5,6$, the scheme is unconditionally unstable for $d=-1$, and the maximum distance value for such polynomials is much lower: $d=-0.25$ for $\mathbb{P}^5$, and $d=-0.05$ for $\mathbb{P}^6$.
Also in this case, there are no cases where it is possible to stabilize the ROD-$L^2$ method by lowering the CFL. ROD-$L^2$ is either stable under a classical CFL condition or unconditionally unstable. Similarly to what was observed for ROD-E.

\begin{table}
        \caption{Linear advection equation: $L_2$ errors $\|u-u_{ex}\|_2$ and estimated order of accuracy (EOA) with the ROD-$L^2$ correction and explicit time integration for polynomials $\mathbb{P}^p$.
        CFL and $d$ are chosen to consider the worst case scenario, i.e.\ the furthest possible $d\in[-1,0]$ and the associated maximum CFL.
        }\label{tab:LAE_RODL2explicit}
        \scriptsize
        \centering
        \begin{tabular}{c|cc|cc|cc} \hline
                        &\multicolumn{2}{c|}{$\mathbb{P}^1$, CFL=1.0, $d=-1.00$} &\multicolumn{2}{c|}{$\mathbb{P}^2$, CFL=1.0, $d=-1.00$} &\multicolumn{2}{c}{$\mathbb{P}^3$, CFL=1.0, $d=-1.00$}  \\[0.5mm]
                        \cline{2-7}
                        $N_e$,  & $L_2$ error & EOA  & $L_2$  error & EOA   & $L_2$ error & EOA    \\ \hline 
                        20      &  5.57E-04  &  --   & 1.92E-03  &  --  & 1.19E-05 &  --  \\
                        40      &  1.24E-04  &  2.16 & 2.40E-04  & 2.99 & 3.74E-07 & 4.99   \\
                        80      &  3.00E-05  &  2.04 & 3.01E-05  & 2.99 & 1.17E-08 & 4.99   \\
                        160     &  7.44E-06  &  2.01 & 3.76E-06  & 2.99 & 3.68E-10 & 4.99   \\
        \end{tabular}
        \begin{tabular}{c|cc|cc|cc} \hline
                        &\multicolumn{2}{c|}{$\mathbb{P}^4$, CFL=1.0, $d=-1.00$} &\multicolumn{2}{c|}{$\mathbb{P}^5$, CFL=1.0, $d=-0.25$} &\multicolumn{2}{c}{$\mathbb{P}^6$, CFL=1.0, $d=-0.05$}  \\[0.5mm]
                        \cline{2-7}
                        $N_e$,  & $L_2$ error & EOA  & $L_2$  error & EOA   & $L_2$ error & EOA    \\ \hline 
                        5      &  1.93E-02  &  --  &  2.42E-05   &  --  & 1.40E-07  &  --   \\
                        10     &  6.46E-04  & 4.90 &  1.98E-07   & 6.93 & 1.27E-09  & 6.78  \\
                        20     &  2.05E-05  & 4.97 &  1.57E-09   & 6.98 & 1.03E-11  & 6.94  \\
                        40     &  6.44E-07  & 4.99 &  1.23E-11   & 6.99 & 8.26E-14  & 6.95  \\ \hline
        \end{tabular}
\end{table}

In table \ref{tab:LAE_RODEimplicit}, we present the numerical results obtained with the ROD-E correction and implicit time integration. 
For implicit time integration, it is always possible to obtain a stable scheme for $d=-1$ but for some cases the considered method has a lower bound condition on the CFL.
For the ROD-E method with $p=1,2,3$ there are no conditions on the CFL, and it is therefore possible to use any CFL value and the scheme remains stable. This was expected since these schemes were also stable under the standard CFL condition with explicit time integration. In this case, it is interesting to see if coupling higher order methods with implicit time integration can result in stable schemes. 
However, for $p=4,5,6$ a lower bound on the CFL is necessary to ensure stability: for $p=4$ it is necessary to have $\text{CFL}\geq 3$, for $p=5$ we need $\text{CFL}\geq 6$, while for $p=6$ we have $\text{CFL}\geq 9$. 

\begin{table}
        \caption{Linear advection equation: $L_2$ errors $\|u-u_{ex}\|_2$ and estimated order of accuracy (EOA) with ROD-E correction and implicit time integration for polynomials $\mathbb{P}^p$.
        It is always considered the case where $d=-1.0$ with the smallest CFL $\geq 1$ value to make the scheme stable.
        }\label{tab:LAE_RODEimplicit}
        \scriptsize
        \centering
        \begin{tabular}{c|cc|cc|cc} \hline
                        &\multicolumn{2}{c|}{$\mathbb{P}^1$, CFL=1.0} &\multicolumn{2}{c|}{$\mathbb{P}^2$, CFL=1.0} &\multicolumn{2}{c}{$\mathbb{P}^3$, CFL=1.0}  \\[0.5mm]
                        \cline{2-7}
                        $N_e$,  & $L_2$ error & EOA  & $L_2$  error & EOA   & $L_2$ error & EOA    \\ \hline 
                        20      & 5.08E-04  &  --  &  1.45E-03  &  --  & 2.28E-05 &    --  \\
                        40      & 1.21E-04  & 2.07 &  1.81E-04  & 2.99 & 7.16E-07 &  4.99  \\
                        80      & 2.98E-05  & 2.01 &  2.26E-05  & 2.99 & 2.23E-08 &  5.00  \\
                        160     & 7.43E-06  & 2.00 &  2.83E-06  & 2.99 & 6.28E-10 &  5.14  \\
        \end{tabular}
        \begin{tabular}{c|cc|cc|cc} \hline
                        &\multicolumn{2}{c|}{$\mathbb{P}^4$, CFL=3.0} &\multicolumn{2}{c|}{$\mathbb{P}^5$, CFL=6.0} &\multicolumn{2}{c}{$\mathbb{P}^6$, CFL=9.0}  \\[0.5mm]
                        \cline{2-7}
                        $N_e$,  & $L_2$ error & EOA  & $L_2$  error & EOA   & $L_2$ error & EOA    \\ \hline 
                        5       & 2.77E-02  &  --  &  2.50E-02    &  --  &  2.33E-02   &  --   \\
                        10      & 1.10E-03  & 4.65 &  2.08E-04    & 6.91 &  2.17E-04   &  6.74 \\
                        20      & 3.62E-05  & 4.92 &  1.65E-06    & 6.97 &  1.77E-06   &  6.93 \\
                        40      & 1.15E-06  & 4.98 &  1.29E-08    & 6.99 &  1.44E-08   &  6.93 \\ \hline
        \end{tabular}
\end{table}

In table \ref{tab:LAE_RODL2implicit}, we present the numerical results obtained with the ROD-$L^2$ correction and implicit time integration. 
Also in this case, it is always possible to obtain a stable scheme for $d=-1$. However, in some cases, the method has a lower bound condition on the CFL.
For the ROD-$L^2$ method with $p=1,2,3,4$ there are no restruction on the CFL, and the scheme will stay stable. This was expected since these schemes were also stable under standard CFL condition with explicit time integration. 
Once again, we study if coupling higher order methods with implicit time integration can improve the stability for $p=5,6$. 
We show that for $p=5$ it is necessary to have $\text{CFL}\geq 0.7$, while for $p=6$ we need $\text{CFL}\geq 2$. 
Although less severe than for ROD-E, the ROD-$L^2$ method also exhibits parasitic modes for very high order polynomials that require higher CFL values for stabilization.

\begin{table}
        \caption{Linear advection equation: $L_2$ errors $\|u-u_{ex}\|_2$ and estimated order of accuracy (EOA) with ROD-$L^2$ correction and implicit time integration for polynomials $\mathbb{P}^p$.
        It is always considered the case where $d=-1.0$ with the smallest CFL $\geq 1$ value to make the scheme stable.
        }\label{tab:LAE_RODL2implicit}
        \scriptsize
        \centering
        \begin{tabular}{c|cc|cc|cc} \hline
                        &\multicolumn{2}{c|}{$\mathbb{P}^1$, CFL=1.0} &\multicolumn{2}{c|}{$\mathbb{P}^2$, CFL=1.0} &\multicolumn{2}{c}{$\mathbb{P}^3$, CFL=1.0}  \\[0.5mm]
                        \cline{2-7}
                        $N_e$,  & $L_2$ error & EOA  & $L_2$  error & EOA   & $L_2$ error & EOA    \\ \hline 
                        20      &  5.57E-04 &  --  &  1.92E-03  &  --  & 1.19E-05 & --    \\
                        40      &  1.24E-04 & 2.16 &  2.40E-04  & 2.99 & 3.74E-07 &  4.99  \\
                        80      &  3.00E-05 & 2.04 &  3.01E-05  & 2.99 & 1.17E-08 &  4.99  \\
                        160     &  7.44E-06 & 2.01 &  3.76E-06  & 2.99 & 3.68E-10 &  4.99  \\
        \end{tabular}
        \begin{tabular}{c|cc|cc|cc} \hline
                        &\multicolumn{2}{c|}{$\mathbb{P}^4$, CFL=1.0} &\multicolumn{2}{c|}{$\mathbb{P}^5$, CFL=1.0} &\multicolumn{2}{c}{$\mathbb{P}^6$, CFL=2.0}  \\[0.5mm]
                        \cline{2-7}
                        $N_e$,  & $L_2$ error & EOA  & $L_2$  error & EOA   & $L_2$ error & EOA    \\ \hline 
                        5       & 1.93E-02  &  --  &   2.06E-03   &  --    &   1.53E-03  &   --  \\
                        10      & 6.46E-04  & 4.90 &   1.67E-05   &  6.94  &   1.29E-05  &  6.88 \\
                        20      & 2.05E-05  & 4.97 &   1.32E-07   &  6.98  &   1.03E-07  &  6.97 \\
                        40      & 6.44E-07  & 4.99 &   1.03E-09   &  6.99  &   8.09E-10  &  6.99 \\ \hline
        \end{tabular}
\end{table}

\section{Conclusions}\label{sec:conclusions}

Herein, we presented a thorough stability analysis for DG methods coupled with arbitrary high order ROD embedded treatment, in the context of hyperbolic equations.
The numerical study was performed by visualizing the eigenspectrum of the discretized operators for a simplified system, 
and has provided critical insights into the stability properties of the considered geometrically unfitted approaches.
To perform this study, it was necessary to write all embedded methods in matrix form, to allows for a straightforward integration into the global discretized operators.  
In this paper, we show for the first time that the ROD method can be recast as a polynomial correction that only depends on the basis function values on the real boundary and on the computational one, similarly to the one obtained for the SB method in \cite{ciallella2023shifted}.  
For all ROD versions, it is interesting to notice the stability region is not symmetric when considering external ($d<0$) or internal ($d>0$) embedded boundaries.
The findings of the analysis show that the ROD polynomial correction, when discarding the element crossed by the embedded boundary (i.e.\ for $d<0$ configurations), does not introduce any restriction on the CFL for polynomial degree $p\leq 3$ for the ROD-E version, and $p\leq 4$ for the ROD-$L^2$ version. 
This is an advantage with respect to the SB method, which shows a restriction on the CFL that becomes stricter as the polynomial degree increases \cite{ciallella2025stability}.
Instead, for $d>0$, all considered methods become unconditionally unstable beyond a critical distance $d>d_{max}(p)$, where $d_{max}$ decreases when increasing the polynomial degree.
In the second part of the analysis, we study the impact of implicit time integration to understand whether it is possible to overcome the instabilities that arise with explicit methods.
In all cases, we observe improvements in the stability properties of the ROD methods coupled with implicit time integration, and for all configurations it is possible to obtain an unconditionally stable scheme. 
However for very high order methods, several configurations present lower bound conditions on the CFL: i.e.\ the CFL number should be high enough to stabilize the method. This is probably related to the fact that, even for implicit time integration, there are still some parasitic modes that make the method unstable and need larger CFL values to be stabilized.         
An extensive set of numerical experiments for the linear advection equation validate all results obtained by the stability analysis.
The high order convergence slopes were recovered for all stable configurations $(d,\text{CFL})$, and instabilities were observed precisely where the analysis predicted them.

Future work will extend this stability analysis to hyperbolic equations discretized by stabilized continuous Galerkin to investigate the impact of these embedded methods on a different finite element framework.
Further implementations of the efficient ROD polynomial corrections in multiple dimensions will be pursued to improve computational performances on more realistic configurations.
The possibility of perfoming a stability analysis based on the energy stability theory \cite{carpenter1999stable,mattsson2003boundary} and on the G-K-S theory \cite{gustafsson1972stability,strikwerda1980initial} will also be investigated.

\bibliographystyle{plain}
\bibliography{references}

\end{document}